\documentclass[aos,preprint]{imsart}
\RequirePackage[OT1]{fontenc}
\RequirePackage{amsthm,amsmath}
\RequirePackage[numbers]{natbib}
\usepackage{pictexwd}
\usepackage{bm,amssymb,graphicx,url}
\usepackage{color}
\usepackage{epstopdf}
\usepackage[normalem]{ulem}

\usepackage[colorlinks,citecolor=blue,urlcolor=blue]{hyperref}

\startlocaldefs    

\def\th{\theta}
\def\R{\mathbb R}
\def\dd{\Delta}

\def\d{\delta}

\def\G{{\mathbb G}}
\def\H{{\mathbb H}}

\def\l{\lambda}
\def\labda1{\lambda_1}
\def\labda2{\lambda_2}

\def\e{\varepsilon}
\def\f{\phi}

\def\k{\kappa}

\def\s{\sigma}

\def\comment#1{\relax}

\def\=in{\mathop{\rm =}}

\numberwithin{equation}{section}
\theoremstyle{plain}
\endlocaldefs

\usepackage{amsmath,here,amsfonts,latexsym,graphicx,here}
\usepackage{amsthm,curves}
\usepackage{graphicx}
\usepackage{caption}
\usepackage{subcaption}
\usepackage{float}

\def\e{\epsilon}
\def\l{\lambda}

\def\f{\phi}
\def\d{\delta}
\def\dd{\Delta}

\def\th{\theta}
\def\s{\sigma}

\def\G{{\mathbb G}}
\def\H{{\mathbb H}}
\def\Q{{\mathbb Q}}
\def\R{{\mathbb R}}

\def\bmy{\bm y}

\newtheorem{theorem}{Theorem}
\newtheorem{lemma}{Lemma}

\newtheorem{remark}{Remark}

\newtheorem{example}{Example}

\begin{document}
\begin{frontmatter}
\title{Nonparametric Least squares estimators for interval censoring}
\runtitle{LS estimators for interval censoring}

\begin{aug}
\author{\fnms{Piet} \snm{Groeneboom}\corref{}\ead[label=e1]{P.Groeneboom@tudelft.nl}}
\runauthor{Piet Groeneboom}
\address{Delft Institute of Applied Mathematics, Mekelweg 4, 2628 CD Delft,
	The Netherlands.\\ 
	\printead{e1}} 
\end{aug}

\begin{abstract}
The limit distribution of the nonparametric maximum likelihood estimator for interval censored data with more than one observation time per unobservable observation, is still unknown in general. For the so-called separated case, where one has  observation times which are at a distance larger than a fixed epsilon>0, the limit distribution was derived in [5]. For the non-separated case there is a conjectured limit distribution, given in [10], Section 5.2 of Part 2. Whether this conjecture holds is still unknown, but the present paper shows that for sample sizes 1000 and 10,000 this limit behavior is still not clearly seen.

We prove consistency of a related nonparametric isotonic least squares estimator and sketch of the proof for its limit distribution. We also provide simulation results to show how the nonparametric MLE and least squares estimator behave in comparison. Moreover, we discuss a simpler least squares estimator that can be computed in one step, but is inferior to the other least squares estimator, since it does not use all information.

 For the simplest model of interval censoring, the current status model, the nonparametric maximum likelihood and least squares estimators are the same. This equivalence breaks down if there are more observation times per unobservable observation. The computations for the simulation of the more complicated interval censoring model were performed by using the iterative convex minorant algorithm. They are provided in the GitHub repository [7].
\end{abstract} 

\begin{keyword}[class=AMS]
\kwd[Primary ]{62G05}
\kwd{62N01}
\kwd[; secondary ]{62-04}
\end{keyword}

\begin{keyword}
\kwd{interval censoring}
\kwd{nonparametric least squares estimators}
\kwd{nonparametric maximum likelihood}
\kwd{asymptotic distribution theory}
\kwd{integral equations}
\kwd{smooth functional theory}
\kwd{iterative convex minorant algorithm}
\end{keyword}

\end{frontmatter}

\section{Introduction}
\label{sec:intro}
The simplest and most studied interval censoring model is the so-called interval censoring, case 1, or current status model.
This model can be defined in the following way (see, e.g., Section 2.3 of \cite{piet_geurt:14}).

Consider a sample $X_1,X_2,\ldots,X_n$, drawn from a distribution with distribution function $F_0$. Instead of observing the $X_i$'s, one only observes for each $i$ whether or not $X_i\le T_i$ for some random $T_i$ (independent of the other $T_j$'s and all $X_j$'s). More formally, instead of observing the $X_i$'s, one observes
\begin{equation}
\label{eq:obsIC1}
(T_i,\Delta_i)=(T_i,1_{\{X_i\le T_i\}}).
\end{equation}
One could say that the $i$-th observation represents the current status of item $i$ at time $T_i$.

We want to estimate the unknown distribution function $F_0$ based on the data given in (\ref{eq:obsIC1}). For this problem the log likelihood function for distribution functions $F$ (conditional on the $T_i$'s) is given by
\begin{equation}
\label{eq:loglikICI}
\ell(F)=\sum_{i=1}^n \left(\dd_i\log F(T_i)+(1-\dd_i)\log(1-F(T_i))\right).
\end{equation}
The (nonparametric) maximum likelihood estimator $\hat F_n$ maximizes $\ell$ over the class of {\it all} distribution functions.

Since distribution functions are by definition nondecreasing, computing the maximum likelihood estimator poses a shape constrained optimization problem. As can be seen from (\ref{eq:loglikICI}), the value of $\ell$ only depends on the values that $F$ takes at the observed time points $T_i$. Hence one can choose to consider only distribution functions that are constant between successive observed time points $T_i$. The lemma below shows that  this estimator can be characterized in terms of a greatest convex minorant of a certain diagram of points. The following result is Lemma 2.7 in \cite{piet_geurt:14} (and Proposition 1.2 in \cite{GrWe:92}).

\begin{lemma}
\label{lem:charMLECS}[Characterization of the nonparametric ML estimator in the current status model]
Consider the cumulative sum (cusum) diagram consisting of the points $P_0=(0,0)$ and
$$
P_i=\left(i,\sum_{j=1}^i\dd_j\right),\,\,\,1\le i\le n,
$$
where the $\dd_i$'s correspond to the $T_i$'s, which are supposed to be ordered $0<T_1\dots< T_n$ (one can also allow ties, but we disregard this further complication here). 
Then the nonparametric MLE $\hat F_n(T_i)$ is given by the left derivative of the convex minorant of this diagram of points, evaluated at the point $i$. This maximizer is unique.
\end{lemma}

\begin{remark}
{\rm
The {\it left} derivative of the convex minorant at $P_i$ determines the value of $\hat{F}_n$ at $T_i$ and hence (by right continuity of the distribution function $\hat F_n$) on $[T_i,T_{i+1})$, a region to the {\it right} of $T_i$.
}
\end{remark}

Lemma \ref{lem:charMLECS} shows that $\hat F_n$ is in fact the {\it isotonic regression} on the indicators $\dd_i=1_{\{X_i T_i\}}$, that is, it minimizes
\begin{align}
\label{LS_CS}
\sum_{i=1}^n \{F(T_i)-\dd_i\}^2
\end{align}
over all monotone nondecreasing (not necessarily bounded by $0$ and $1$) functions $F$, see p.\ 43 in section 1.1 of part 2 of \cite{GrWe:92}. It also follows from Theorem 1.2.1 on p.\ 7 of \cite{rwd:88}, where the connection brtween the derivative of the greatest convex minorant and the solution of the isotonic regression problem is given.

So we have two characterizations of the nonparametric maximum likelihood estimator, the characterization as a maximizer of (\ref{eq:loglikICI}) and the characterization as a minimizer of (\ref{LS_CS}). Note that the weights are equal and constant in the least sauares problem.
 
If we turn to the (common) situation where there are more observation times per unobservable $X_i$, the situation is considerably more complicated, and the limit distribution of the nonparametric MLE is still unknown generally. We consider here the simplest extension, where one has two observation times per unobservable $X_i$ ({\it interval censoring, case 2}). Instead of observing the $X_i$'s, one observes
\begin{equation}
\label{eq:interval2}
(U_i,V_i,\Delta_{i0},\Delta_{i1})\stackrel{\text{def}}=(U_i,V_i,1_{\{X_i\le U_i\}},1_{\{U_i<X_i\le V_i\}}),
\end{equation}
so instead of only observing that $X_i$ is larger than or smaller than an observation $T_i$, we now have an observation interval $(U_i,V_i)$, $U_i<V_i$,  and we know whether our unobservable $X_i$ is inside the interval or to the left or right of it. Just as in the current status model, we assume that $X_i$ is distributed independently of $(U_i,V_i)$. The log likelihood (\ref{eq:loglikICI}) changes into
\begin{align}
\label{eq:loglik_IC2}
\ell(F)=\sum_{i=1}^n \left\{\dd_{i0}\log F(U_i)+\dd_{i1}\log(F(V_i)-F(U_i))+\dd_{i2}\log(1-F(V_i))\right\},
\end{align}
where $\dd_{i2}=1-\dd_{i0}-\dd_{i1}$.
It is not clear that there is an equivalent to the minimization of (\ref{LS_CS}) in this situation. We cannot use Theorem 1.5.1 in \cite{rwd:88}  because we have the difference of $F(V_i)$ and $F(U_i)$ in the terms $\log(F(V_i)-F(U_i))$ instead of just $F(U_i)$ or $F(V_i)$ by itself only.

The perhaps most natural least squares approach is to consider minimization of
\begin{align}
\label{LS_criterion_IC}
&\sum_{i=1}^n \left\{\left\{F(U_i)-\dd_{i0}\right\}^2+ \left\{F(V_i)-F(U_i)-\dd_{i1}\right\}^2
+\left\{1-F(V_i)-\dd_{i2}\right\}^2\right\}
\end{align}
over all distribution functions $F$, so an isotonic regression on the indicators $\dd_{i0},\dd_{i1},\dd_{i2}$, $i=1,\dots,n$.

As an example, we analyze the behavior of the nonparametric MLE, maximizing (\ref{eq:loglik_IC2}) and the least squares estimator minimizing (\ref{LS_criterion_IC}) in Example \ref{example1}.

 \begin{figure}[!ht]
\begin{subfigure}[b]{0.45\textwidth}
\includegraphics[width=\textwidth]{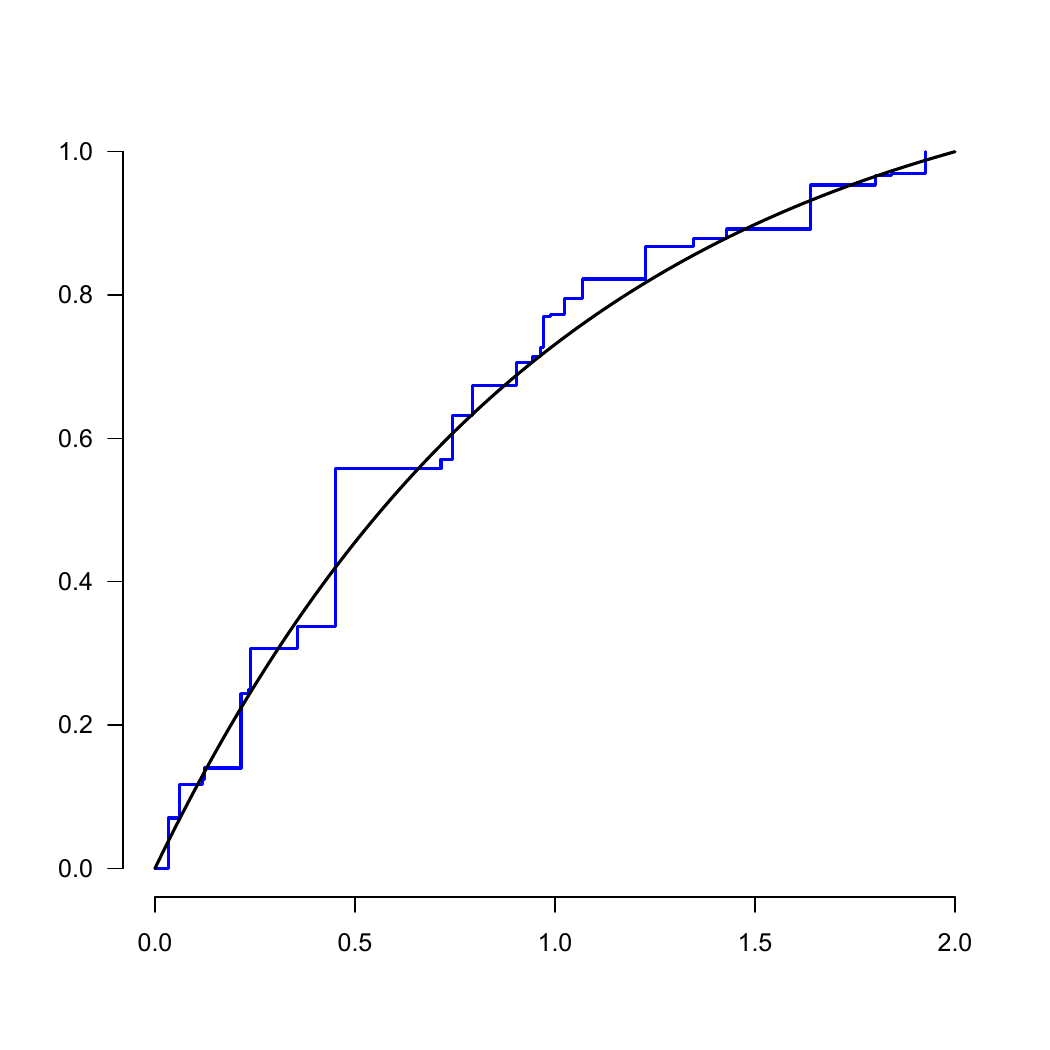}
\caption{}
\label{fig:CMLE1000_IC}
\end{subfigure}
\begin{subfigure}[b]{0.45\textwidth}
\includegraphics[width=\textwidth]{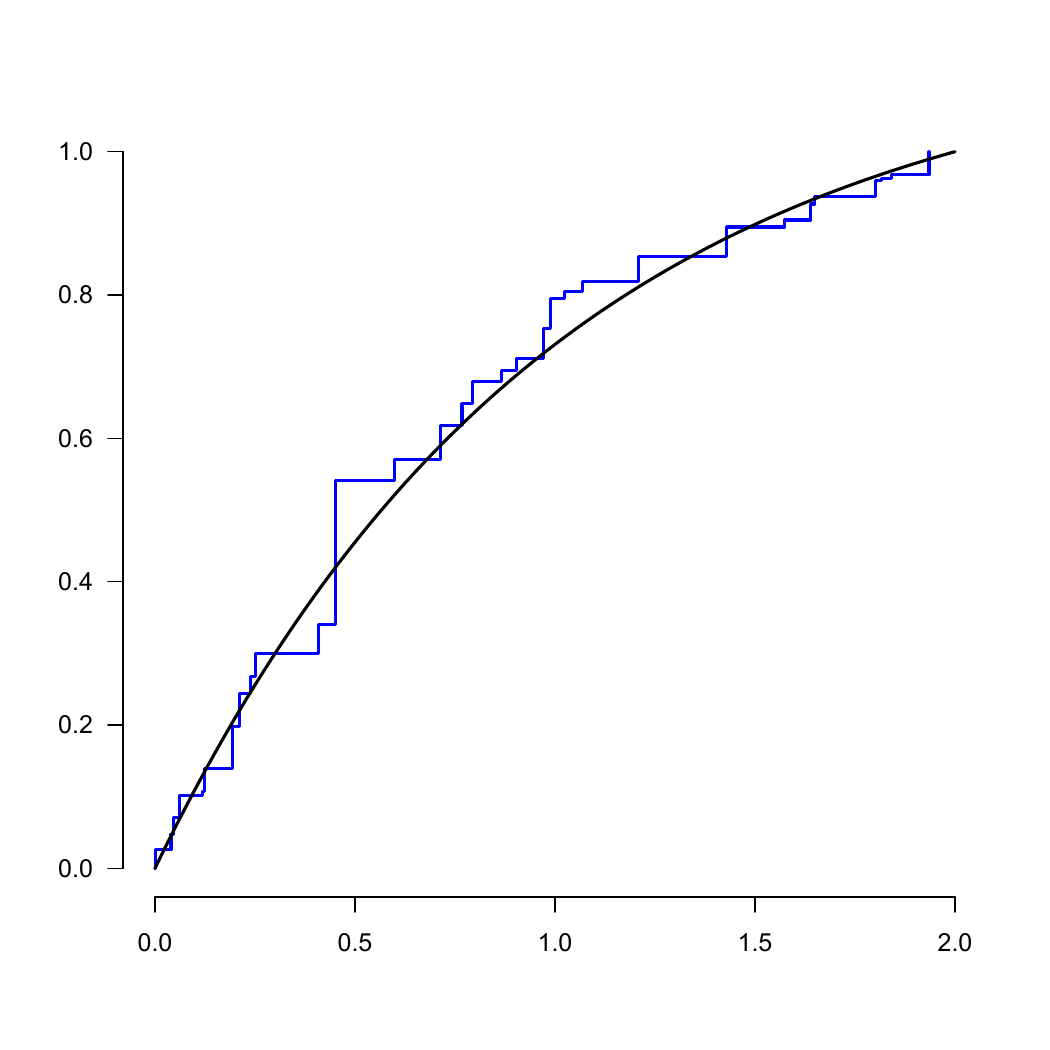}
\caption{}
\label{fig:LS1000_IC}
\end{subfigure}
\caption{(a) Nonparametric MLE (blue) of $F_0$ for a sample of size $n=1000$, (b) Nonparametric least squares estimate (blue) minimizing (\ref{LS_criterion_IC}) of $F_0$ for the same sample. The solid black curve shows $F_0$.}
\label{figure:MLE+LS}
\end{figure}

\begin{example}
\label{example1}
{\rm Suppose $X_1,\dots,X_n$ is either a sample from the truncated exponential distribution on $[0,2]$, with density
\begin{align*}
f_0(x)=\frac{\exp\{-x\}}{1-\exp\{-2\}}1_{[0,2]}(x),\qquad x\in\R,
\end{align*}
or a sample from the Uniform distribution on $[0,2]$,
and let the $(U_i,V_i),\,V_i>U_i$, be the order statistics of the Uniform distribution on $[0,2]^2$ for a sample of size 2. Note that this is a prototype of the non-separated case, where we can have arbitrarily small observation intervals $[U_i,V_i]$

\begin{figure}[!ht]
\begin{subfigure}[b]{0.45\textwidth}
\includegraphics[width=\textwidth]{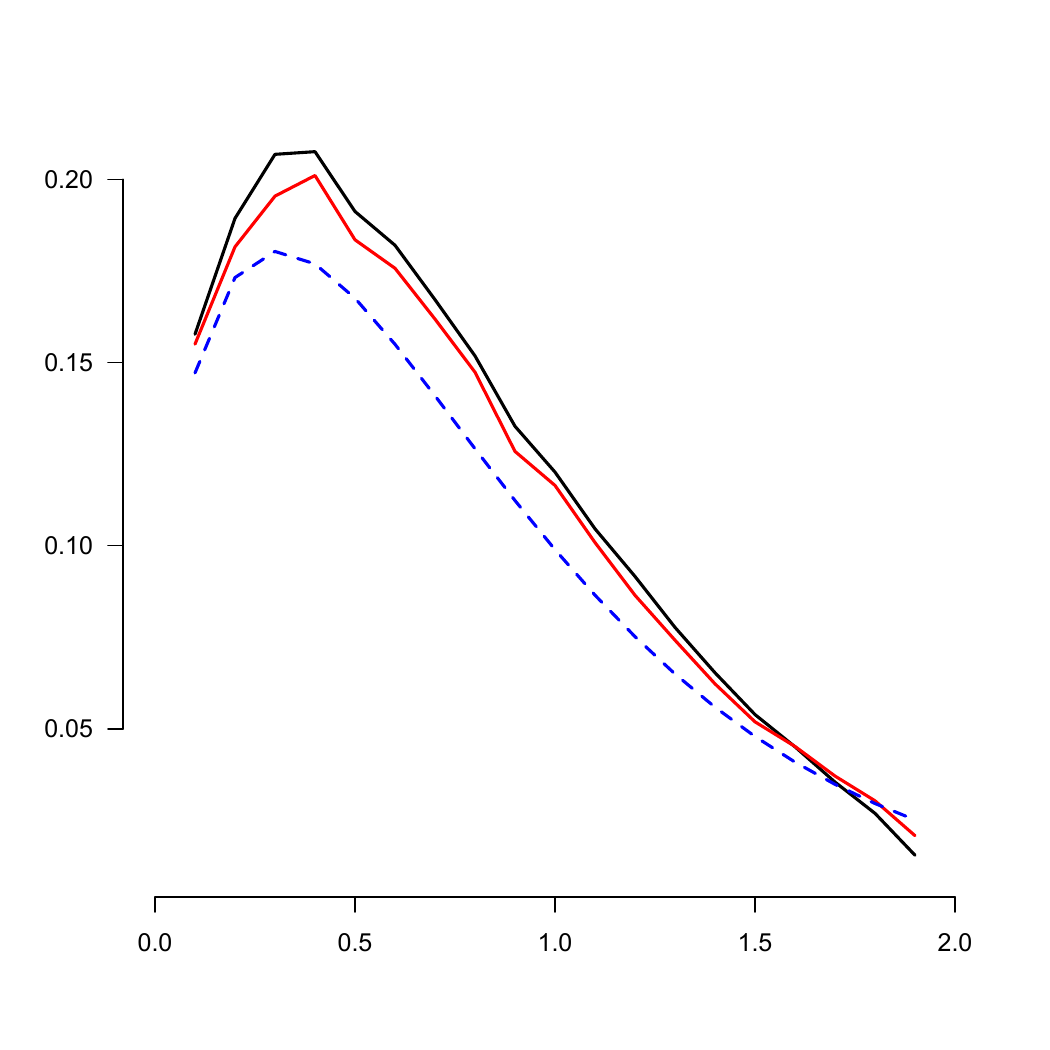}
\caption{}
\label{fig:CMLE1000_exponential}
\end{subfigure}
\begin{subfigure}[b]{0.45\textwidth}
\includegraphics[width=\textwidth]{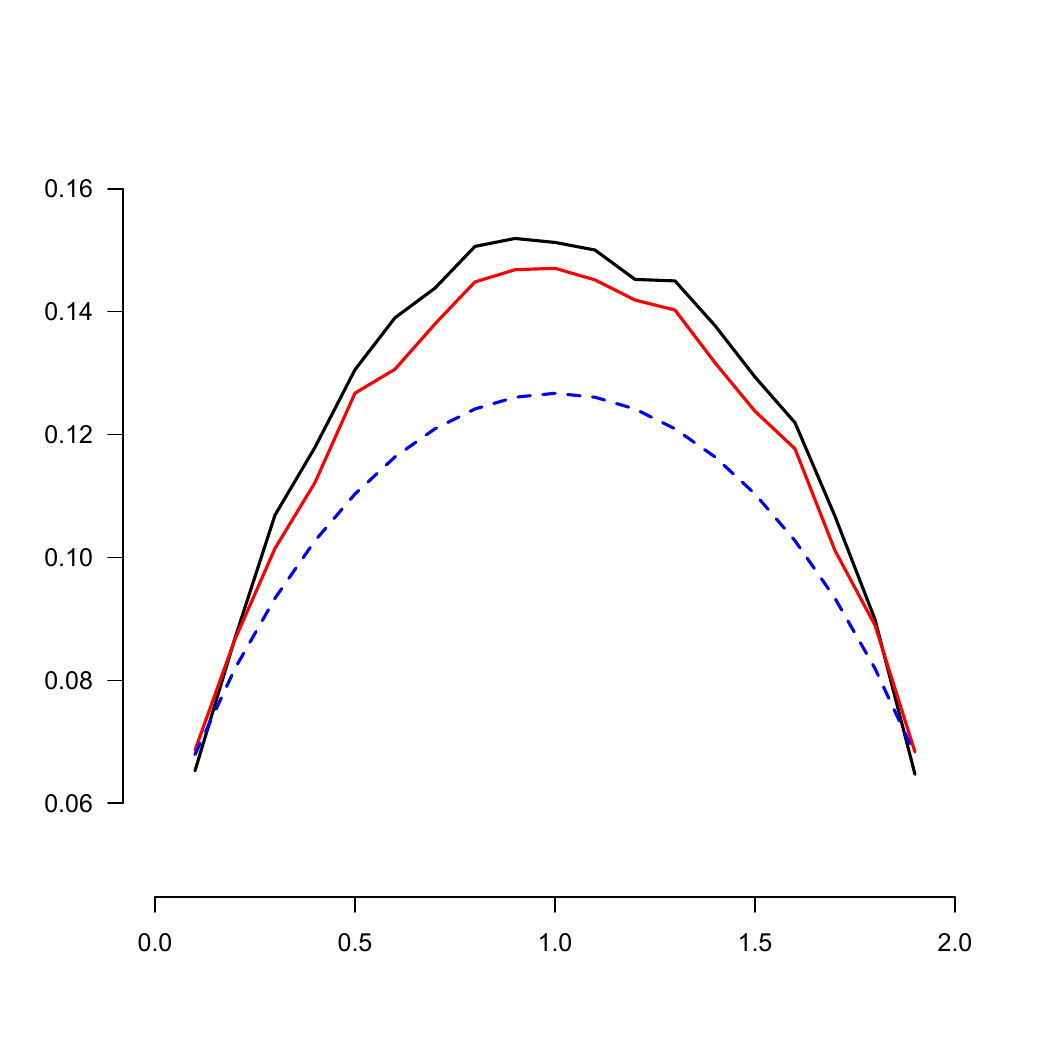}
\caption{}
\label{fig:LS1000_MLE_uniform}
\end{subfigure}
\caption{(a) Simulated variances, times $n^{2/3}$, of the nonparametric MLE (black solid curve) and the least squares estimate minimizing (\ref{LS_criterion_IC}) (red)), for $t_i=0.1,0.2,\dots,1.9$, linearly interpolated between values at the $t_i$ for the model of Example \ref{example1}. The blue dashed curve is the theoretical limit curve one obtains from Theorem \ref{th:limit_LS} in Section \ref{sec:LS} for the LS estimator. The simulated variances are based on $10,000$ simulations of samples of size $n=1000$ for the truncated exponential distribution function $F_0$ on $[0,2]$ and the order statistics of the uniform distribution on $[0,2]^2$ as observation times. (b) The same comparison, but now for $F_0$ uniform on $[0,2]$.}
\label{figure:MLE+LS_variances}
\end{figure}

Figure \ref{figure:MLE+LS} shows the nonparametric MLE and the nonparametric least squares estimate for a simulated sample of size $n=1000$ for these models. In both cases the estimate has to be computed iteratively, we do not have a one step algorithm as in the current status model.

We computed the variances of the estimates times $n^{2/3}$ for $10,000$ samples of size $n=1000$ for $t_i=i\cdot 0.1\,$, $i=1,\dots,19$. A comparison of the simulated variances is shown in Figure \ref{figure:MLE+LS_variances}. It suggests that overall the least squares estimator is slightly better than the nonparametric MLE for this sample size in this model.
 }
 \end{example}

 Inspired by the conjectured faster rate of convergence of the MLE in the non-separated case (i.e., observation intervals can be arbitrarily small) in Section 5.2 of Part 2 of \cite{GrWe:92}, Lucien Birg\'e constructed in \cite{lucien:99} a histogram estimator which actually achieves rate $(n\log n)^{1/3}$ locally (but suffers severely from bias, in contrast with the MLE). As noted above, we still do not know whether the MLE also achieves this faster rate.
  
 It is also possible to define a least squares estimator which can be computed in one step, using the convex minorant (or ``pool adjacent violators'') algorithm. An estimator of this type was proposed for the uniform deconvolution problem in \cite{eszuijlen:96}. This estimator minimizes
\begin{align}
\label{LS_criterion_IC2a}
&\sum_{i=1}^n \left\{\left\{F(U_i)-\dd_{i0}\right\}^2+\left\{F(V_i)-\dd_{i0}-\dd_{i1}\right\}^2\right\}
\end{align}
over all distribution functions $F$, so is also an isotonic regression on the indicators $\dd_{i0},\dd_{i1},\dd_{i2}$, $i=1,\dots,n$. Criterion (\ref{LS_criterion_IC2a}) can also be written
\begin{align}
\label{LS_criterion_IC3}
&\sum_{i=1}^n \left\{\left\{F(U_i)-\dd_{i0}\right\}^2+\left\{1-F(V_i)-\dd_{i2}\right\}^2\right\}.
\end{align}
We compare the two least squares estimators for sample size $n=10,000$ in Figure \ref{figure:two_LS_estimators10,000}, from which the superiority of the estimator, based on minimizing (\ref{LS_criterion_IC}), seems clear. This second least squares estimator is in fact very much of the same type as the LS estimator/MLE in the current status model; it only checks whether the unobservable observation is larger than or smalle than a real observation time.
\begin{figure}[!ht]
\begin{subfigure}[b]{0.4\textwidth}
\includegraphics[width=\textwidth]{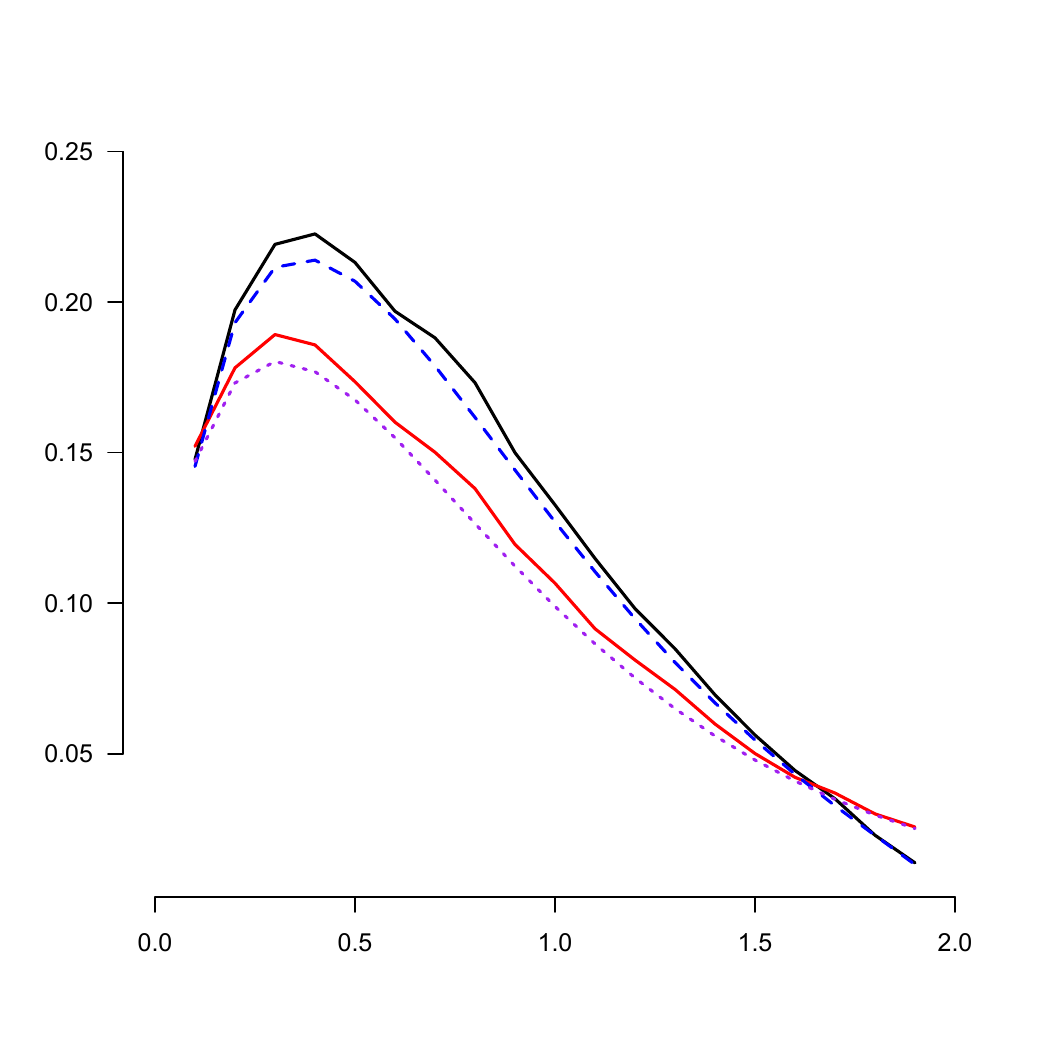}
\caption{}
\label{fig:CMLE10,000_exponential2}
\end{subfigure}
\begin{subfigure}[b]{0.4\textwidth}
\includegraphics[width=\textwidth]{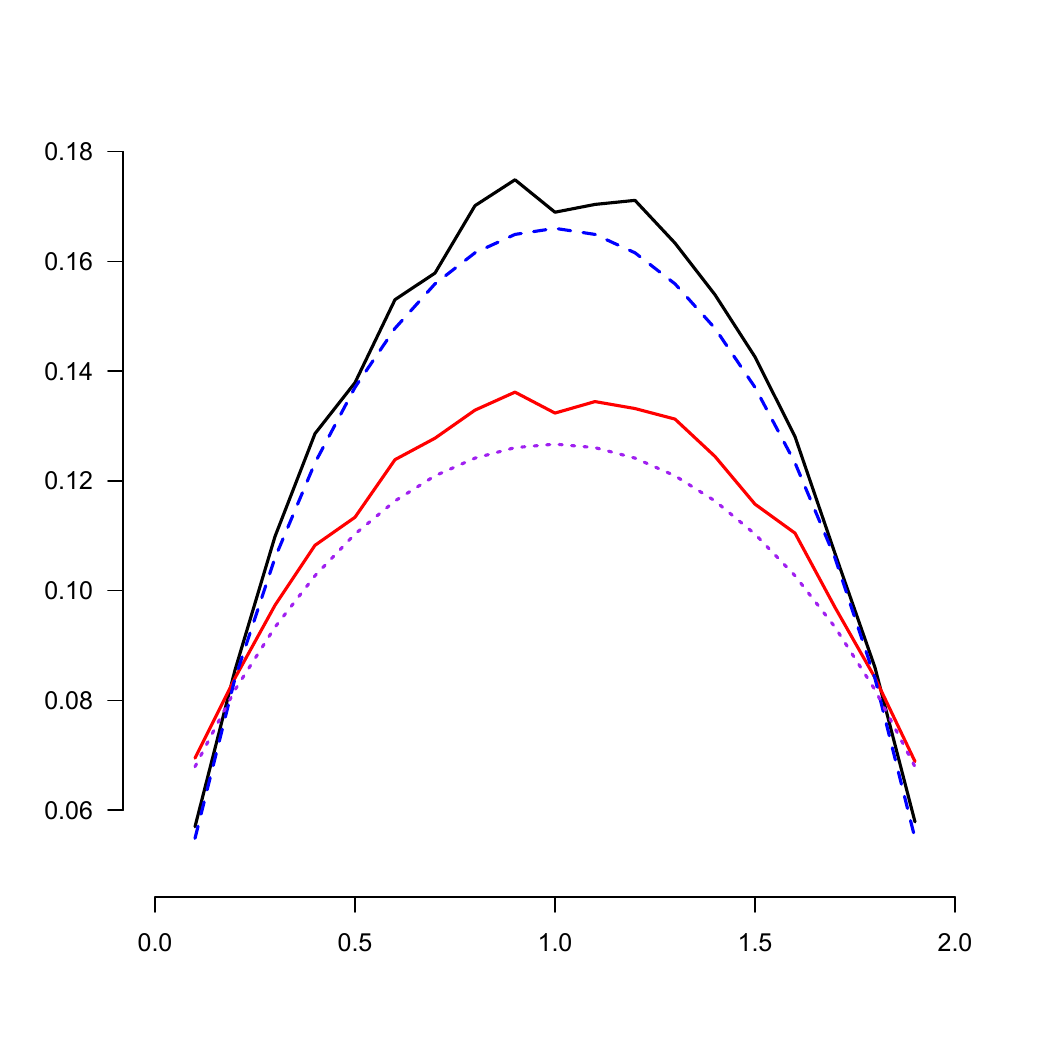}
\caption{}
\label{fig:LS10,000_MLE_uniform2}
\end{subfigure}
\caption{(a) Simulated variances, times $n^{2/3}$, of the simple nonparametric LS estimator, minimizing (\ref{LS_criterion_IC2a}) (black solid curve) and the least squares estimator, minimizing (\ref{LS_criterion_IC}) (red), for $t_i=0.1,0.2,\dots,1.9$, linearly interpolated between values at the $t_i$ for the model of Example \ref{example1}. The blue dashed curve and purple dotted curves are the theoretical limit curves discussed in Section \ref{sec:LS} for the LS estimators, minimizing (\ref{LS_criterion_IC2a}) and (\ref{LS_criterion_IC}), respectively The simulated variances are based on $10,000$ simulations of samples of size $n=10,000$ for the truncated exponential distribution function $F_0$ on $[0,2]$ and the order statistics of the uniform distribution on $[0,2]^2$ as observation times. (b) The same comparison, but now for $F_0$ uniform on $[0,2]$.}
\label{figure:two_LS_estimators10,000}
\end{figure}

Our proof of the limit result Theorem \ref{th:limit_LS} is not complete. We still have to show that the ``off-diagonal terms'' (\ref{off-diag1}) and (\ref{off-diag2}) are $o_p(n^{-2/3})$, see Section \ref{appendix2}. In fact, terms of this type were shown to be $O_p(n^{-5/6})$ in a similar problem, analyzed in \cite{piet:24}. In Section \ref{section:smooth_functionals} we discuss the theory, needed for the completion of the proof.

The proof of Theorem \ref{th:limit_LS2} on the limit behavior of the second LS estimator is complete, and in fact very much like the proof for the limit behavior of the MLE in the current status model.

Both LS estimators lead to $\sqrt{n}$-consistent estimators of smooth functionals, but the asymptotic variance of the MLE attains the information lower bound, as proved in \cite{GeGr:99}, so cannot be beaten here. The influence function of the LS estimators is in fact different from the efficient influence function of the MLE, which can be interpreted as a conditional expectation, see Section \ref{section:smooth_functionals}.

\section{The least squares estimators}
\label{sec:LS}
In the characterization of the nonparametric MLE one can use the fact that the logarithms provide a natural logarithmic boundary, preventing values of the solution to leave the interval $[0,1]$. This is no longer true for the least squares estimate minimizing (\ref{LS_criterion_IC}) and for this reason we use Lagrange multipliers in its characterization.

Let, for a distribution function $F$, the process $W_{n,F}$ be defined by
\begin{align}
\label{def_W_F}
&W_{n,F}(t)\nonumber\\
&=\int_{u\le t}\{\d_0-F(u)\}\,d\Q_n(u,v,\d_0,\d_1)-\int_{u\le t}\{\d_1-\{F(v)-F(u)\}\}\,d\Q_n(u,v,\d_0,\d_1)\nonumber\\
&\quad+\int_{v\le t}\{\d_1-\{F(v)-F(u)\}\}\,d\Q_n(u,v,\d_0,\d_1)+\int_{v\le t}\{\d_0+\d_1-F(v)\}\,d\Q_n(u,v,\d_0,\d_1),
\end{align}
where $\Q_n$ is the empirical probability measure of the $(U_i,V_i,\dd_{i0},\dd_{i1})$.

The least squares estimate of $F_0$ has the following characterization.

\begin{lemma}
\label{lemma:fenchel}
Let the process $W_{n,F}$ be defined by (\ref{def_W_F}) and let the Lagrange multipliers $\l_{1,F}$ and $\l_{2,F}$ be defined by
\begin{align}
\label{lambda_1}
\l_{1,F}=-\int_{(u,v):F(u)=0\,\,\text{\rm or }F(v)=0}\,dW_{n,F}(u,v,\d_0,\d_1),
\end{align}
and
\begin{align}
\label{lambda_2}
\l_{2,F}=\int_{(u,v):F(u)=1\,\,\text{\rm or }F(v)=1}\,dW_{n,F}(u,v,\d_0,\d_1).
\end{align}

Then the distribution function $\hat F_n$ minimizes (\ref{LS_criterion_IC}) over all distribution functions $F$ on $\R_+$ if and only if the following conditions are satisfied
\begin{enumerate}
\item[(i)]
\begin{align*}
\l_{1,\hat F_n}+W_{n,\hat F_n}(t)\ge0,\qquad t\ge0,
\end{align*}
\item[(ii)]
\begin{align*}
\int \hat F_n(t)\,dW_{n,\hat F_n}(t)-\l_{2,\hat F_n}=0.
\end{align*}
\end{enumerate}
\end{lemma}

\begin{figure}[!ht]
\centering
\includegraphics[width=0.5\textwidth]{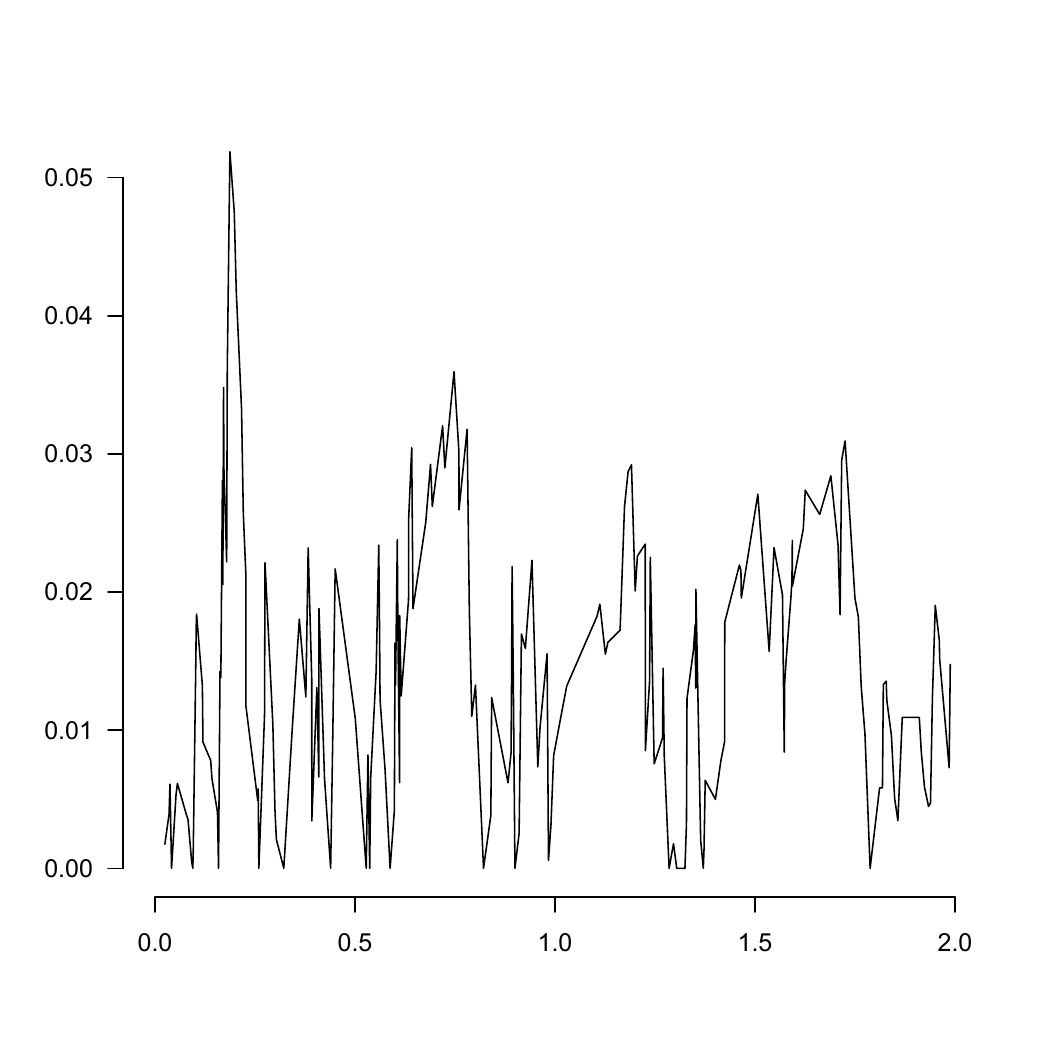}
\caption{The process $W_{n,\hat F_n}$ as a function of the $2n$ ordered observations $U_i$ and $V_i$ for Example \ref{example1} and $n=100$. For this example $\l_{1,\hat F_n}=0.003148$ and $\l_{2,\hat F_n}=0.014758$.}
\label{figure:W}
\end{figure}

The proof is based on the so-called Fenchel duality condition, as also used in \cite{piet_geurt:14} in the characterization of the nonparametric MLE for interval censoring (apart from the Lagrange multipliers in the present case) and is therefore omitted. A picture of the process $W_{n,\hat F_n}$ for $n=100$ is given in Figure \ref{figure:W} for Example \ref{example1} above. $W_{n,\hat F_n}$ touches zero at points just to the left of points where the least squares estimator $\hat F_n$ has a jump. The iterative convex minorant algorithm for computing $\hat F_n$ is based on this characterization.

The algorithm proceeds in the following way.  At each step we solve the isotonic regression problem without the Lagrange multipliers, under the restriction $y_1\le\dots\le y_{2n}$, where $y_i=F(T_i)$ represents the value of the distribution function at the $i$th order statistic of set of $U_i$'s and $V_i$'s (ties can also be handled, but we disregard this further complication here).   Typically, this will give values $y_i<0$ and $y_i>1$. If $y_i<0$ we put $y_i=0$ and if $y_i>1$ we put $y_i=1$. For this new value of the $y_i$ we compute the (preliminary) Lagrange multipliers $\l_1$ and $\l_2$, using (\ref{lambda_1}) and (\ref{lambda_2}). We repeat this procedure until the conditions (i) and (ii) of Lemma \ref{lemma:fenchel} are satisfied up to an accuracy of say $10^{-8}$. Convergence of this algorithm is very fast.

One can also compute the estimate by the interior point method, using a logarithmic barrier function. This algorithm is a kind of opposite of the iterative convex minnorant algorihm, since it converges to the solution from the interior of the parameter space, whereas the iterative convex minorant algorithm immediately hits the boundary in the first iteration step. Both algorithms were programmed for the present problem and give exactly the same result, though.

In order to state the limit result of the least squares estimator, minimizing (\ref{LS_criterion_IC}), we need the following notation.
Let, for $t_0\in(0,M)$,  $a_{t_0}$ be defined as the positive square root of

\begin{align}
\label{scale_IC}
a_{t_0}^2&=F_0(t_0)\{1-F_0(t_0)\}\bigl\{h_1(t_0)+h_2(t_0)\bigr\}\nonumber\\
&\qquad+\int_{v=t_0}^M\{F_0(v)-F_0(t_0)\}\left[1-\{F_0(v)-F_0(t_0)\}\right]\,h(t_0,v)\,dv\nonumber\\
&\qquad+\int_{u=0}^{t_0}\{F_0(t_0)-F_0(u)\}\left[1-\{F_0(t_0)-F_0(u)\}\right]\,h(u,t_0)\,du\nonumber\\
&\qquad+2F_0(t_0)\int_{v=t_0}^M\{F_0(v)-F_0(t_0)\}\,h(t_0,v)\,dv\nonumber\\
&\qquad+2\{1-F_0(t_0)\}\int_{u=0}^{t_0}\{F_0(t_0)-F_0(u)\}\,h(u,t_0)\,du,
\end{align}
and let $b_{t_0}$ be defined by:
\begin{align}
\label{drift_IC}
b_{t_0}&=h_1(t_0)+h_2(t_0),
\end{align}
where $h_1$ and $h_2$ are the marginals of $h$.
Then we have the following result of which the proof is sketched in the Appendix. 

\begin{theorem}
\label{th:limit_LS}
Let $F_0$ have a continuous strictly positive density $f_0$ on $[0,M]$, where $f_0$ stays away from zero,  and let $(U,V)$ be the order statistics of an absolutely  continuous distribution. We assume that $(U,V)$ has a a positive continuous density $h$ on
\begin{align*}
S=\{(u,v):0\le u<v\le M\},
\end{align*}
staying away from zero, with first and second marginals $h_1$ and $h_2$,  respectively, and with bounded  partial derivatives. We assume that $X_i$ is independent of $(U_i,V_i)$. 

Let $a_{t_0}$ and $b_{t_0}$ be defined by (\ref{scale_IC}) and (\ref{drift_IC}), respectively. Moreover let, for $t_0\in(0,M)$,  $\s_{t_0}$ be defined by
\begin{align}
\label{sigma_{t_0}}
\s_{t_0}=\left(a_{t_0}f_0(t_0)/b_{t_0}\right)^{2/3},\qquad t_0\in(0,M).
\end{align}
Then we get at a fixed point $t_0\in(0,M)$ for the LS estimate $\hat F_n$, minimizing
\begin{align}
\label{LS_criterion_IC2}
&\sum_{i=1}^n \left\{\left\{F(U_i)-\dd_{i0}\right\}^2+ \left\{F(V_i)-F(U_i)-\dd_{i1}\right\}^2+\left\{F(V_i)-\dd_{i0}-\dd_{i1})\right\}^2\right\}
\end{align}
over all distribution functions $F$:
\begin{align*}
n^{1/3}\left\{\hat F_n(t_0)-F_0(t_0)\right\}/\s_{t_0}\stackrel{{\cal D}}\longrightarrow Z,
\end{align*}
where $Z$ is the argmin of $t\mapsto W(t)+t^2$, and $W$ is standard two-sided Brownian motion.
\end{theorem}

\begin{remark}
{\rm
The conditions on the observation density $h$ are similar to the conditions in \cite{GeGr:99} and used in the smooth functional theory needed in treating so-called ``off-diagonal terms'', see Section \ref{appendix2}.
}
\end{remark}

The least squares estimator minimizing (\ref{LS_criterion_IC}) is consistent, as the following lemma shows.

\begin{lemma}
\label{lemma:consistency}
Let $\hat F_n$ be the isotonic least squares estimator, minimizing (\ref{LS_criterion_IC}) over distribution functions $F$ under the conditions of Theorem \ref{th:limit_LS}. Then $\hat F_n$ converges almost surely to $F_0$ in the supremum metric.
\end{lemma}

\begin{proof}
We use ``Jewell's method'' (\cite{jewell:82}). Let the function $\psi$ be defined by
\begin{align}
\label{def_psi}
&\psi(F)\nonumber\\
&=\tfrac12\int\left\{\{F(u)-\d_0\}^2+\{F(v)-F(u)-\d_1\}^2+\{F(v)-\d_0-\d_1\}^2\right\}\,d\Q_n(u,v,\d_0,\d_1).
\end{align}
Since $\hat F_n$ minimizes (\ref{def_psi}) over $F$, we must have
\begin{align*}
\lim_{\e\downarrow0}\e^{-1}\left[\psi\bigl((1-\e)\hat F_n+\e F_0\bigr)-\psi(\hat F_n)\right]\ge0.
\end{align*}
Note that the limit exists by the convexity of the function $\psi$.
This means
\begin{align*}
&\int\{F_0(u)-\hat F_n(u)\}\{\hat F_n(u)-\d_0\}\,d\Q_n(u,v,\d_0,\d_1)\\
&\qquad+\int\{F_0(v)-F_0(u)-\hat F_n(v)+\hat F_n(u)\}\{\hat F_n(v)-\hat F_n(u)-\d_1\}\,d\Q_n(u,v,\d_0,\d_1)\\
&\qquad+\int\{F_0(v)-\hat F_n(u)\}\{\hat F_n(v)-\d_0-\d_1\}\,d\Q_n(u,v,\d_0,\d_1)\\
&\ge0.
\end{align*}
Proceeding as in Section 4 of Part 2 of \cite{GrWe:92}, using the Helly compactness theorem, we get from this, for a limit point $F$ of a subsequence of $\hat F_n$, where $h$ is the density of $(U,V)$ (see Theorem \ref{th:limit_LS}):
\begin{align*}
&\int\{F_0(u)-F(u)\}\{F(u)-F_0(u)\}\,h(u,v)\,du\,dv\\
&\qquad+\int\{F_0(v)-F_0(u)-F_(v)-F(v)+F(u)\}\\
&\qquad\qquad\qquad\qquad\qquad\qquad\cdot \{F(v)-F(u)-F_0(v)+F_0(u)\}\,h(u,v)\,du\,dv\\
&\qquad+\int\{F_0(v)-F(v)\}\{F(v)-F_0(v)\}\,h(u,v)\,du\,dv\\
&=-\int\{F(u)-F_0(u)\}^2\,h(u,v)\,du\,dv-\int\{F(v)-F(u)-F_0(v)+F_0(u)\}^2\,h(u,v)\,du\,dv\\
&\qquad-\int\{F(v)-F_0(v)\}^2\,h(u,v)\,du\,dv\\
&\ge0.
\end{align*}
By the assumptions on $F_0$ and $h$ this means that $F=F_0$. This implies that $\hat F_n$ converges uniformly to $F_0$.
\end{proof}

\begin{figure}[!ht]
\begin{subfigure}[b]{0.4\textwidth}
\includegraphics[width=\textwidth]{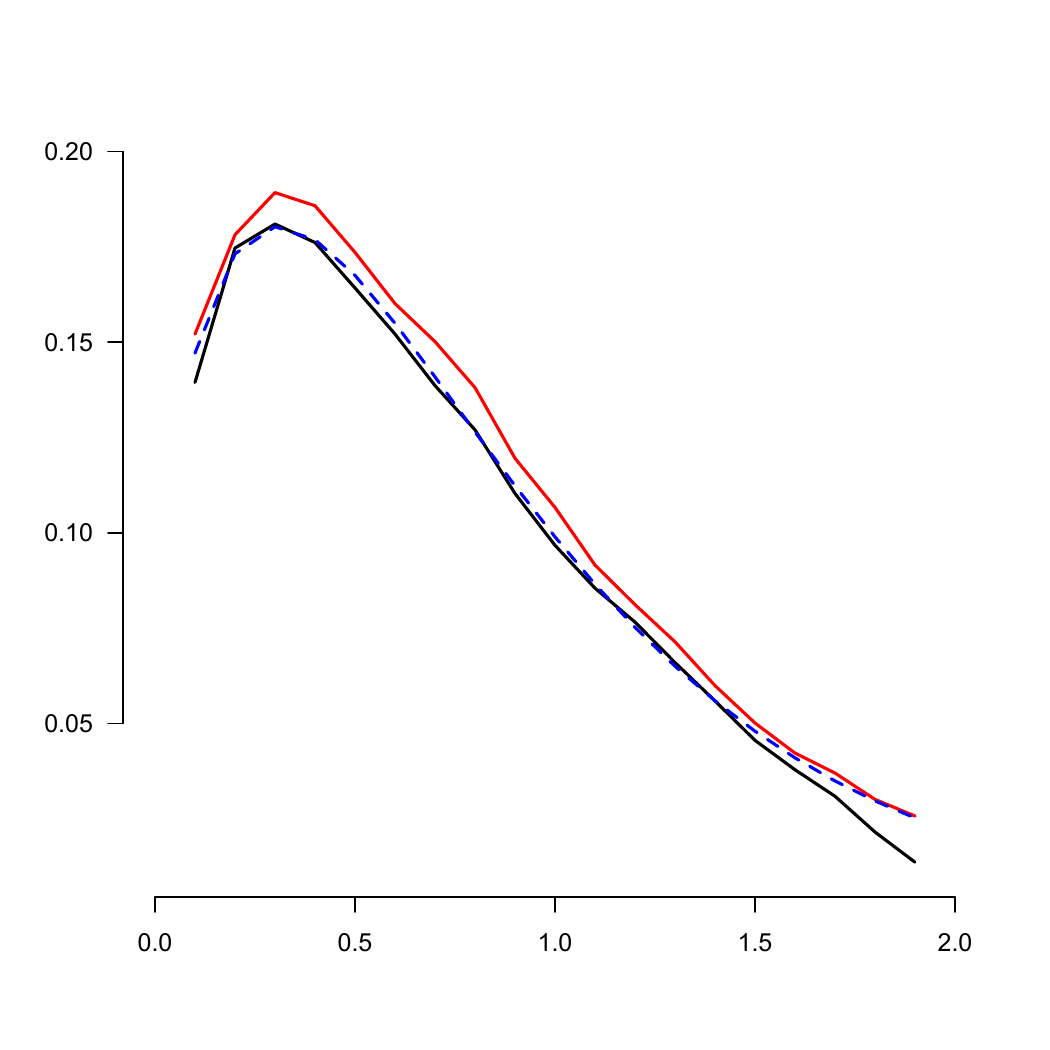}
\caption{}
\label{fig:CMLE10,000_exponential}
\end{subfigure}
\begin{subfigure}[b]{0.4\textwidth}
\includegraphics[width=\textwidth]{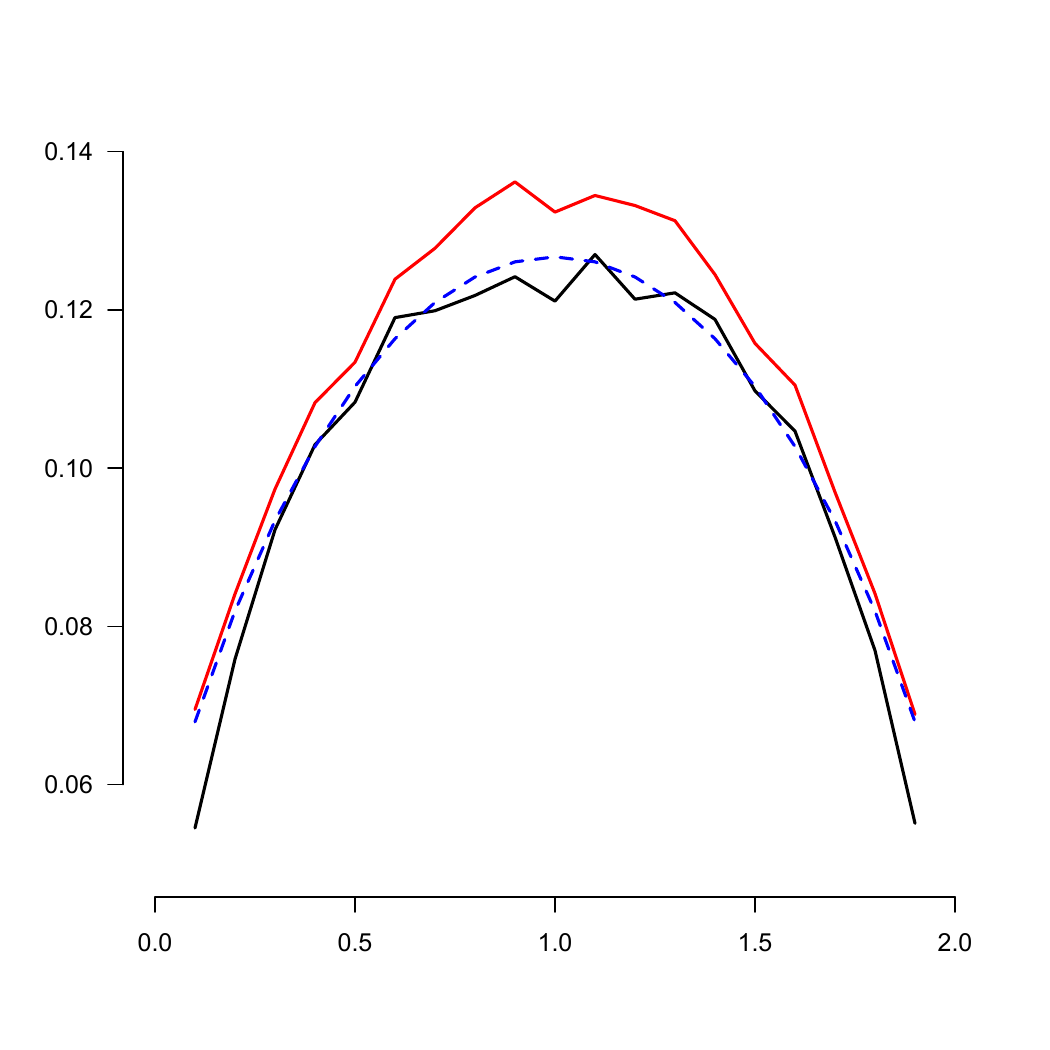}
\caption{}
\label{fig:LS10,000_MLE_uniform}
\end{subfigure}
\caption{(a) Simulated variances, times $n^{2/3}$, of the nonparametric MLE (black solid curve) and the least squares estimate (red), minimizing (\ref{LS_criterion_IC}), for $t_i=0.1,0.2,\dots,1.9$, linearly interpolated between values at the $t_i$ for the model of Example \ref{example1}. The blue dashed curve is the theoretical limit curve one obtains from Theorem \ref{th:limit_LS} below. The simulated variances are based on $10,000$ simulations of samples of size $n=10,000$ for the truncated exponential distribution function $F_0$ on $[0,2]$ and the order statistics of the uniform distribution on $[0,2]^2$ as observation times. (b) The same comparison, but now for $F_0$ uniform on $[0,2]$.}
\label{figure:MLE+LS_variances10,000}
\end{figure}

As an example, the asymptotic variance of the LS estimator is shown in Figure \ref{figure:MLE+LS_variances10,000} as the blue dashed curve for  in Example \ref{example1}. Remarkably, the curve of the variances of the MLE is even closer to the theoretical asymptotic curve than the curve for the variances of the least squares estimator for this sampe size of $n=10,000$. This should be compared with Figure \ref{figure:MLE+LS_variances}, where the sample size was $n=1000$ and the variances for the LS estimator were actually smaller that those of the nonparametric MLE. The faster decrease of the variances of the MLE than those of the LS when we increase the sample size from $1000$ to $10,000$ might give a hint of the possibility of a rate of $(n\log n)^{1/3}$ again, but the form of the curve is certainly not in accordance with the conjecture in \cite{GrWe:92}.

We here make some additional remarks about the proof of Theorem \ref{th:limit_LS}, discussed in the Appendix. Results of this type were proved for the nonparametric MLE in \cite{piet:96} and \cite{piet:24}. One has to show that the leading asymptotic behavior is provided by replacing the estimator by the underlying distribution function $F_0$ in the characterization, given in Lemma \ref{lemma:fenchel}. To this end, one has to show that certain terms, called the ``off-diagonal terms'',  are of lower order. Doing this is rather hard in our experience (\cite{piet:24} needed 53 pages), but presently there seems to be no other way. We give a sketch of these calculations in Section \ref{section:smooth_functionals} and in the Appendix. For similar calculations, see, e.g., Chapter 10 of \cite{piet_geurt:14} and \cite{piet:24}. The first result of the type of Theorem \ref{th:limit_LS} was for the current status model in Theorem 5.5 of Part 2 of \cite{GrWe:92}.

For the least squares estimator, minimizing (\ref{LS_criterion_IC2a}) we can define the following process.
\begin{align}
\label{def_W_F2}
&W_{n,F}^{(2)}(t)\nonumber\\
&=\int_{u\le t}\{\d_0-F(u)\}\,d\Q_n(u,v,\d_0,\d_1)+\int_{v\le t}\{\d_0+\d_1-F(v)\}\,d\Q_n(u,v,\d_0,\d_1),\,t\ge0,
\end{align}
where $\Q_n$ is the empirical probability measre of the $(U_i,V_i,\dd_{i0},\dd_{i1})$ and $F$ a distribution function. We now get the following characterization the least squares estimator, minimizing (\ref{LS_criterion_IC2a}).

\begin{lemma}
\label{lemma:fenchel2}
Let the process $W_{n,F}^{(2)}$ be defined by (\ref{def_W_F2}).
Then the distribution function $\hat F_n$ minimizes (\ref{LS_criterion_IC2a}) over all distribution functions $F$ on $\R_+$ if and only if the following conditions are satisfied
\begin{enumerate}
\item[(i)]
\begin{align*}
W^{(2)}_{n,\hat F_n}(t)\ge0,\qquad t\ge0,
\end{align*}
\item[(ii)]
\begin{align*}
\int \hat F_n(t)\,dW^{(2)}_{n,\hat F_n}(t)=0.
\end{align*}
\end{enumerate}
\end{lemma}

The proof follows familiar lines and is therefore omitted. Alternatively, one can say that the esitimator is the left-continuous slope of the greatest convex minorant of the cusum diagram starting at $(0,0)$ and running through the ($2n$) points
\begin{align*}
\left(\sum 1_{\{U_i\le t\}}+\sum 1_{\{V_i\le t\}},\sum\dd_{i0}1_{\{U_i\le t\}}+\sum(\dd_{i0}+\dd_{i1})1_{\{V_i\le t\}}\right),\qquad t\ge0.
\end{align*}
Since we work with observations from absolutely continuous distributions, we get for the left coordinates  just $1,2,\dots$.

 So, in constructing the cusum diagram, we run through all $2n$ observation points; if we meet a point $U_i$, we add $\d_{i0}$ to the cusum diagram,  if we meet a point $V_i$ we add $\d_{i0}+\d_{i1}$ to the diagram. In the latter case we only record whether the unobservable variable lies to the left of $V_i$ (then $\d_{i0}+\d_{i1}=1$) or to the right of $V_i$ (then $\d_{i0}+\d_{i1}=0$), so we do not use the information it lies between $U_i$ and $V_i$ if this is the case.

For the isotonic least squares estimator, minimizing (\ref{LS_criterion_IC2a}) we have the following result.
\begin{theorem}
\label{th:limit_LS2}
Let the assumptions of Theorem \ref{th:limit_LS} be satisfied and let $a_{t_0}'$ and $b_{t_0}'$ be defined by
\begin{align}
\label{scale_IC2}
a_{t_0}'&=\left\{F_0(t_0)\{1-F_0(t_0)\}\bigl\{h_1(t_0)+h_2(t_0)\bigr\}\right\}^{1/2},
\end{align}
and
\begin{align}
\label{drift_IC2}
b_{t_0}'&=\left\{h_1(t_0)+h_2(t_0)\right\}/2,
\end{align}
where $h_1$ and $h_2$ are the marginals of $h$.
Moreover let, for $t_0\in(0,M)$,  $\s_{t_0}$ be defined by
\begin{align}
\label{sigma_{t_0}2}
\s_{t_0}'=\left(a_{t_0}'f_0(t_0)/b_{t_0}'\right)^{2/3},\qquad t_0\in(0,M).
\end{align}
Then we get at a fixed point $t_0\in(0,M)$ for the LS estimate $\hat F_n$, minimizing
\begin{align}
\label{LS_criterion_IC2b}
&\sum_{i=1}^n \left\{\left\{F(U_i)-\dd_{i0}\right\}^2+\left\{F(V_i)-\dd_{i0}-\dd_{i1})\right\}^2\right\}
\end{align}
over all distribution functions $F$:
\begin{align*}
n^{1/3}\left\{\hat F_n(t_0)-F_0(t_0)\right\}/\s_{t_0}'\stackrel{{\cal D}}\longrightarrow Z,
\end{align*}
where $Z$ is the argmin of $t\mapsto W(t)+t^2$, and $W$ is standard two-sided Brownian motion.
\end{theorem}

The proof is given in the Appendix.
 Both the computation and the proof of Theorem \ref{th:limit_LS2}  are easier than for the estimator minimizing (\ref{LS_criterion_IC2}), but the performance of the estimator, minimizing (\ref{LS_criterion_IC2a}) is clearly inferior to that of the estimator, minimizing (\ref{LS_criterion_IC}), see Figure \ref{figure:two_LS_estimators10,000}.
 
 It is instructive to compare the results of Theorems \ref{th:limit_LS} and \ref{th:limit_LS2} with the limit result for the MLE (which is also an LS estimator with equal constant weights) for the current status model. The result is given as Theorem 3.7 in \cite{piet_geurt:14} on p.\ 68. In that case we have
 \begin{align*}
 n^{1/3}\left\{\hat F_n(t_0)-F_0(t_0)\right\}/\s \stackrel{{\cal D}}\longrightarrow Z,
 \end{align*}
 where $Z=\text{argmin}_t\{W(t)+t^2\}$ and $\s=(a/b)^{2/3}$, with
 \begin{align*}
 a=\left\{F_0(t_0)\{1-F_0(t_0)\}g(t_0)f_0(t_0)\right\}^{1/2}\qquad\text{ and }\qquad b=\tfrac12g(t_0).
 \end{align*}
 Theorem \ref{th:limit_LS2} is completely similar, with $g(t_0)$ replaced by $h_1(t_0)+h_2(t_0)$ and  Theorem \ref{th:limit_LS} also has this structure, but in that case $g(t_0)$ is replaced by $2\{h_1(t_0)+h_2(t_0)\}$ and $a$ is replaced by a more complicated scale factor, reflecting the covariances of the trinomial distribution of $\dd_{i0},\,\dd_{i1}$ and $\dd_{i2}$.
 
 \begin{figure}[!ht]
\begin{subfigure}[b]{0.45\textwidth}
\includegraphics[width=\textwidth]{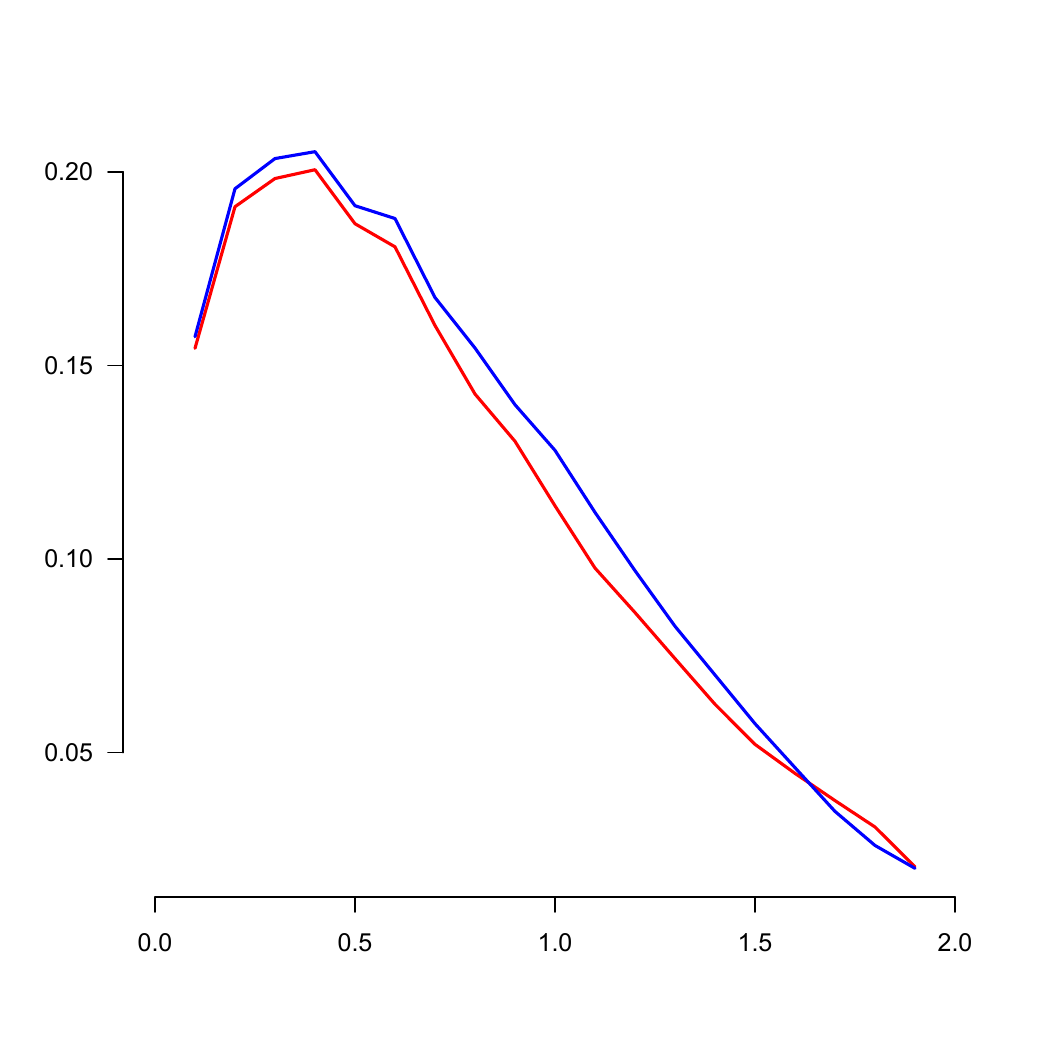}
\caption{}
\label{fig:CMLE1000_exponential2}
\end{subfigure}
\begin{subfigure}[b]{0.45\textwidth}
\includegraphics[width=\textwidth]{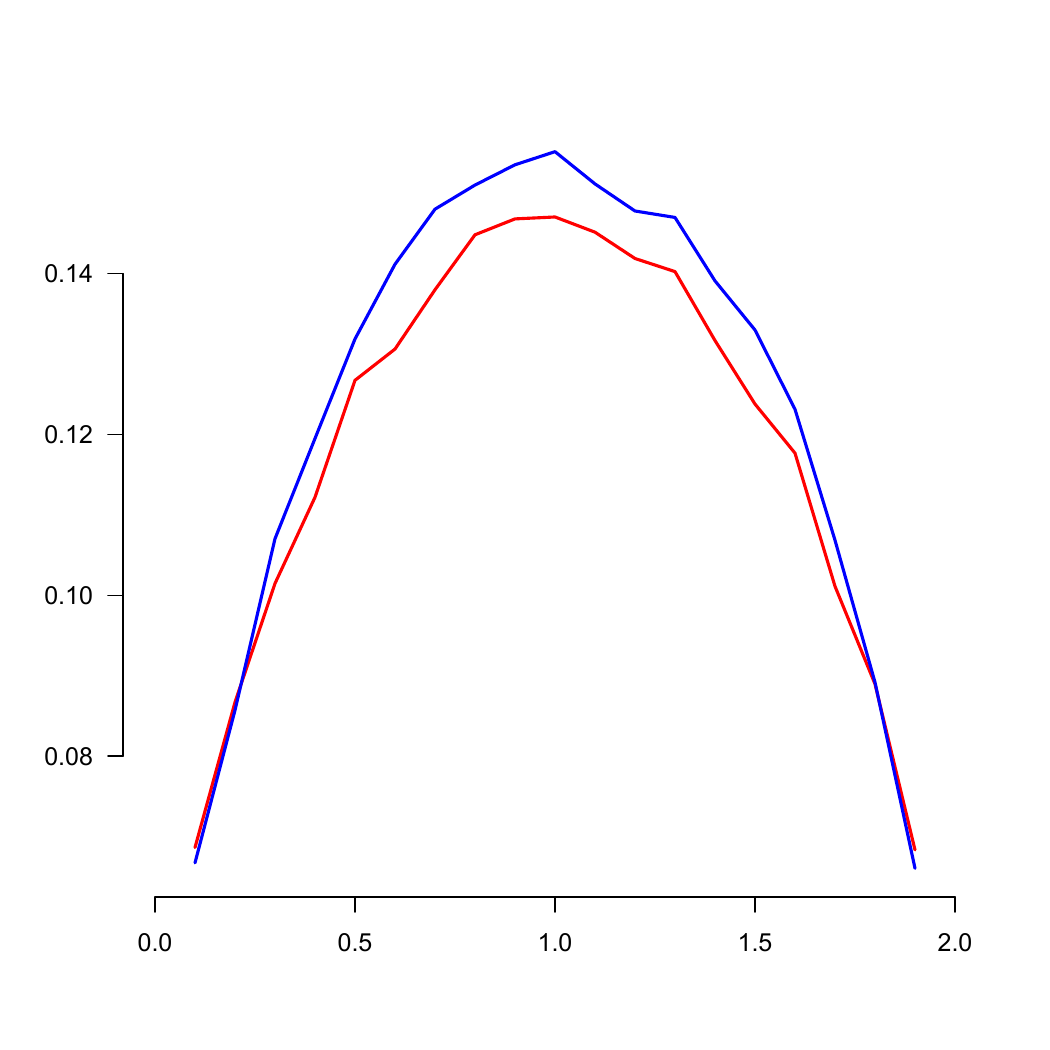}
\caption{}
\label{fig:LS1000_MLE_uniform2}
\end{subfigure}
\caption{(a) Simulated variances, times $n^{2/3}$, of the LS estimator, minimizing (\ref{LS_criterion_IC}) (red curve) and the LS estimator minimizing (\ref{LS_criterion_IC2}) (blue), for $t_i=0.1,0.2,\dots,1.9$, linearly interpolated between values at the $t_i$ for the model of Example \ref{example1}. The simulated variances are based on $10,000$ simulations of samples of size $n=1000$ for the truncated exponential distribution function $F_0$ on $[0,2]$ and the order statistics of the uniform distribution on $[0,2]^2$ as observation times. (b) The same comparison, but now for $F_0$ uniform on $[0,2]$.}
\label{figure:LS+LSweighted_variances}
\end{figure}

We chose equal weights, because the MLE for current status data is equivalent to an LS estimator with equal weights. But it is also possible to introduce other weights for the LS estimators. We experimented a bit with weights of the form $w_{i0}=1/U_i$, $w_{i1}=1/(V_i-U_i)$ and $w_{i2}=1/(2-V_i)$, $i=1,\dots,n$, for our examples of the truncated exponential distribution and the Uniform distribution on $[0,2]$, giving smaller intervals more weight.  Our findings were that the iterative procedure for computing the LS estimator, minimizing
 \begin{align}
 \label{LS_criterion_IC_weights}
&\sum_{i=1}^n \left\{\left\{F(U_i)-\dd_{i0}\right\}^2w_{i0}+ \left\{F(V_i)-F(U_i)-\dd_{i1}\right\}^2w_{i1}
+\left\{1-F(V_i)-\dd_{i2}\right\}^2w_{i2}\right\}
\end{align}
lasted longer and became less stable. Also, the results for the algorithm with equal weights were slightly better, for sample size $n=1000$, see Figure \ref{figure:LS+LSweighted_variances}.

 \section{Smooth functional theory}
\label{section:smooth_functionals}
The smooth functional theory for the nonparametric MLE in the separated case is discussed in \cite{piet_geurt:14} and in \cite{piet:96} and for the non-separated case of interval censoring in \cite{GeGr:99}. We give the treatment here for the LS estimator, minimizing (\ref{LS_criterion_IC}).
We define, for a distribution function $F$, the function $\th_{F}$ by:
\begin{align}
\label{theta_F}
\theta_{F}(u,v,\d_0,\d_1)&=\bigl\{\d_0-F(u)\bigr\}\f_{F}(u)+\bigl\{\d_1-F(v)+F(u)\bigr\}\bigl\{\f_{F}(v)-\f_{F}(u)\bigr\}\nonumber\\
&\qquad\qquad\qquad\qquad\qquad\qquad\qquad\qquad-\bigl\{\d_2-(1-F(v)\}\f_{\hat F_n}(v).
\end{align}
where $\f_{F}$ is either an integrated score function
\begin{align}
\label{def_phi_F}
 \f_{F}(x)=\int_{y\in[0,x]}a(y)\,dF(y),
 \end{align}
 or a function belonging to the closure (in an $L_2$-sense) of the space of such functions (the latter will actually be the case for the functions $\f_{\hat F_n}$, found below for the LS estimator)

Note that, if $\f_{\hat F_n}$ is a right-continuous piecewise constant function, constant on the same intervals as $\hat F_n$, and satisfying $\f_{\hat F_n}(t)=0$  if $\hat F_n(t)=0$ or $\hat F_n(t)=1$, we have:
\begin{align}
\label{LS_score_condition}
&\int\theta_{\hat F_n}(u,v,\d_0,\d_1)\,d\Q_nu,v,\d_0,\d_1)\nonumber\\
&=\int\left[\d_0-\hat F_n(u)-\bigl\{\d_1-\hat F_n(v)+\hat F_n(u)\bigr\}\right]\f_{\hat F_n}(u)\,d\Q_n(u,v,\d_0,\d_1)\nonumber\\
&\qquad\qquad+\int\left[\d_1-\hat F_n(v)+\hat F_n(u)-\d_2+\{1-\hat F_n(v)\bigr\}\right]\f_{\hat F_n}(v)\,d\Q_nu,v,\d_0,\d_1)\nonumber\\
&=0.
\end{align}
We also have:
\begin{align*}
\int\th_{F_0}\,dQ_0=0,
\end{align*}
for the underlying probability measure $Q_0$.

Note that $\th_F$ is similar to, but different from, the canonical gradient or ``efficient influence function'' $\tilde\th=\tilde\th_{F,H}$ in \cite{GeGr:99} (see, e.g., the discussion preceding Theorem 2.1 in \cite{GeGr:99}). The theory in \cite{GeGr:99} shows that the information lower bound $\|\tilde\th\|^2_{Q_0}$, attained by the asymptotic variance of the MLE in estimating smooth functionals, cannot be improved upon, so we can only hope that the corresponding asymptotic variance $\|\th_{F_0}\|^2_{Q_0}$ of the LS estimator (see below) will be close to this information lower bound.

We have, using Fubini's theorem,
\begin{align*}
&\int \th_{\hat F_n}\,dQ_0\\
&=\int\bigl\{F_0(u)-\hat F_n(u)\bigr\}\f_{\hat F_n}(u)\,h_1(u)\,du+\int\bigl\{F_0(v)-\hat F_n(v)\bigr\}\f_{\hat F_n}(v)\,h_2(v)\,dv\\
&\qquad\qquad+\int\bigl\{F_0(v)-F_0(u)-\hat F_n(v)+\hat F_n(u)\bigr\}\bigl\{\f_{\hat F_n}(v)-\f_{\hat F_n}(u)\bigr\}\,h(u,v)\,du\,dv\\
&=\int_{x\in[0,M]}\left[\int_{x\le u}\f_{\hat F_n}(u)\,\{h_1(u)+h_2(u)\}\,du\right.\\
&\qquad\qquad\qquad\qquad\qquad\left.+\int_{u<x\le v}\bigl\{\f_{\hat F_n}(v)-\f_{\hat F_n}(u)\bigr\}h(u,v)\,du\,dv\right]\,d(\hat F_0-\hat F_n)(x).
\end{align*}
If we want to estimate a smooth functional of the type $\int\kappa_{F_0}(x)\,dF_0(x)$, we approximately want to accomplish for the inner part between brackets of this integral:
\begin{align}
\label{adjoint_eq}
&\int_{u\ge x}\f_{\hat F_n}(u)\,\{h_1(u)+h_2(u)\}\,du+\int_{u<x\le v}\bigl\{\f_{\hat F_n}(v)-\f_{\hat F_n}(u)\bigr\}\,h(u,v)\,du\,dv\nonumber\\
&=\k_{F_0}(x),
\end{align}
because we then get
\begin{align*}
\int \k_{F_0}\,d\bigl(\hat F_n-F_0\bigr)\approx \int\th_{\hat F_n}\,dQ_0=\int\th_{\hat F_n}\,d\bigl(Q_0-\Q_n\bigr),
\end{align*}
that is, we can represent the functional in the hidden space by an integral in the observation space which is $\sqrt{n}$ convergent and asymptotically normal.

Differentiating equation (\ref{adjoint_eq}) w.r.t.\ $x$, we get
\begin{align*}
&-\f_{\hat F_n}(x)\{h_1(x)+h_2(x)\}-\int_{u\le x}\bigl\{\f_{\hat F_n}(x)-\f_{\hat F_n}(u)\bigr\}h(u,x)\,du\\
&\qquad\qquad\qquad\qquad\qquad\qquad+\int_{v\ge x}\bigl\{\f_{\hat F_n}(v)-\f_{\hat F_n}(x)\bigr\}h(x,v)\,dv\\
&=\k'_{F_0}(x).
\end{align*}
Defining, as in \cite{GeGr:99}, $\f_{\hat F_n}$ by
\begin{align*}
\f_{\hat F_n}(x)=\int_{y\in(x,M]}a(y)\,d\hat F_n(y)
\end{align*}
instead of $\f_{\hat F_n}(x)=\int_{y\in[0,x]}a(y)\,d\hat F_n(y)$, and also allowing functions $\f_{\hat F_n}$ on the boundary of the space, we get the equation
\begin{align*}
&\f_{\hat F_n}(x)\{h_1(x)+h_2(x)\}+\int_{u\le x}\bigl\{\f_{\hat F_n}(x)-\f_{\hat F_n}(u)\bigr\}h(u,x)\,du\\
&\qquad\qquad\qquad\qquad\qquad\qquad-\int_{v\ge x}\bigl\{\f_{\hat F_n}(v)-\f_{\hat F_n}(x)\bigr\}h(x,v)\,dv\\
&=\k'_{F_0}(x).
\end{align*}

This leads to the matrix equation
\begin{align}
\label{matrix_eq}
&y_i\left\{\dd_i(h_1)+\dd_i(h_2)+\sum_{j<i}\dd_{ji}(h)+\sum_{j>i}\dd_{ij}(h)\right\}\nonumber\\
&=\k_{F_0}(t_{i+1})-\k_{F_0}(t_i)+\sum_{j<i}\dd_{ji}(h)y_j+\sum_{j>i}\dd_{ij}(h)y_j,\qquad i=1,\dots,m-1,
\end{align}
where the $t_i$ are the points of jump of $\hat F_n$ and $y_i=\phi_{\hat F_n}(t_i)$; $\phi_{\hat F_n}(t)=0$ is zero if $\hat F_n(t)=0$ or $\hat F_n(t)=1$.  This is analogous to (iii) of Theorem 3.1 on p.\ 646 of \cite{GeGr:99}. Here $\dd_i(h_j)$ and $\dd_{ij}(h)$ are defined by
\begin{align}
\label{delta_h_j}
\dd_i(h_j) =\int_{t_i}^{t_{i+1}}h_j(t)\,dt, i=1,\dots,m-1,\qquad j=1,2,
\end{align}
and
\begin{align}
\label{delta_h}
\dd_{ij}(h) =\int_{t_i}^{t_{i+1}}\int_{t_j}^{t_{j+1}} h(u,v)\,du\,dv,\qquad 1\le i<j\le m-1.
\end{align}

A prototype of a smooth functional is the functional for the mean
\begin{align*}
T_F:x\mapsto x-\int y\,dF(y).
\end{align*}
For this case we get the equation
\begin{align*}
&-\f_{\hat F_n}(x)\{h_1(x)+h_2(x)\}-\int_{u\le x}\bigl\{\f_{\hat F_n}(x)-\f_{\hat F_n}(u)\bigr\}h(u,x)\,du\\
&\qquad\qquad\qquad\qquad\qquad\qquad+\int_{v\ge x}\bigl\{\f_{\hat F_n}(v)-\f_{\hat F_n}(x)\bigr\}h(x,v)\,dv\\
&=1,\qquad x\in\{z\in(0,M):0<\hat F_n(z)<1\}
\end{align*}

As an example, for the situation that $F_0$ is the uniform distribution function on $[0,1]$ and $(U,V)$ is uniformly distributed on the upper triangle of the unit square with vertices $(0,0)$, $(0,1)$ and $(1,1)$ and, moreover, $\kappa_{F_0}$ is the mean functional, this gives the solution
\begin{align*}
\phi_{\hat F_n}(x)=\left\{\begin{array}{ll} 1/2,\qquad &x\in\{z\in(0,M):0<\hat F_n(z)<1\}\\
0,&\text{otherwise}.
\end{array}
\right.
\end{align*}

This yields as asymptotic variance for $m(\hat F_n)=\int x\,d\hat F_n(x)$, if $\hat F_n$ is the LS estimator, minimizing (\ref{LS_criterion_IC}):
\begin{align*}
&n\,\text{var}\left(m(\hat F_n)\right)\longrightarrow\int\theta_{F_0}(u,v,\d_0,\d_1)^2\,d\Q_0(u,v,\d_0,\d_1),\qquad n\to\infty,
\end{align*}
where
\begin{align*}
&\int\theta_{F_0}(u,v,\d_0,\d_1)^2\,d\Q_0(u,v,\d_0,\d_1)\\
&=\tfrac14\int\left[\d_0-F_0(u)-\d_2+(1-F_0(v))\right]^2\,dQ_0(u,v,\d_0,\d_1)\\
&=\tfrac14\int F_0(u)(1-F_0(u)\}h_1(u)\,du+\tfrac14\int_0^1F_0(v)(1-F_0(v)\}h_2(v)\,dv\\
&\qquad\qquad\qquad\qquad\qquad\qquad\qquad\qquad+\tfrac12\int F_0(u)\{1-F_0(v)\}h(u,v)\,du\,dv\\
&=\tfrac12\int_0^1u(1-u)(1-u)\,du+\tfrac12\int_0^1v(1-v)v\,dv+\int_{0<u<v<1}u(1-v)\,du\,dv\\
&=\frac18=0.125.
\end{align*}
So in this case we get the following result for the LS estimator:
\begin{align*}
\sqrt{n}\left\{m(\hat F_n)-m(F_0)\right\}\stackrel{{\cal D}}\longrightarrow N(0,\s^2),
\end{align*}
where
\begin{align*}
\s^2=\frac18.
\end{align*}

Other smooth functionals can be treated in a similar way, also giving the $\sqrt{n}$ behavior and asymptotic normality of these functionals of the LS estimator, which means that we can follow the analysis in \cite{piet:24} to show that the ``off-diagonal'' terms in Section \ref{appendix2}
\begin{align*}
\int_{u\in[t_0,t_0+n^{-1/3}t],\,v\ge u}\bigl\{F_0(v)-\hat F_n(v)\bigr\}\,dH(u,v).
\end{align*}
and
\begin{align*}
\int_{v\in[t_0,t_0+n^{-1/3}t],\,u\le v}\bigl\{F_0(u)-\hat F_n(u)\bigr\}\,dH(u,v).
\end{align*}
are of order $O_p(n^{-5/6})$.

On the other hand, the analysis of the behavior of the smooth functionals for the nonparametric MLE starts by defining
\begin{align}
\label{theta_F_MLE}
\tilde\theta_{\hat F_n}(u,v,\d_0,\d_1)&=\d_0\frac{\f_{\hat F_n}(u)}{\hat F_n(u)}+\d_1\frac{\f_{\hat F_n}(v)-\f_{\hat F_n}(u)}{\hat F_n(v)-\hat F_n(u)}-\d_2\frac{\f_{\hat F_n}(v)}{(1-\hat F_n(v)}\,,
\end{align}
where $\f_{\hat F_n}$ is defined by (\ref{def_phi_F}), but now for the nonparametric MLE $\hat F_n$. By Proposition 3.1 on p.\ 644 of \cite{GeGr:99} we have:
\begin{align}
\label{MLE_score_condition}
\int\tilde\th_{\hat F_n}\,d\Q_n=0,
\end{align}
(compare with (\ref{LS_score_condition})). Moreover, using Fubini's theorem again, we get
\begin{align*}
&\int \tilde\theta_{\hat F_n}\,dQ_0\\
&=\int F_0(u)\frac{\f_{\hat F_n}(u)}{\hat F_n(u)}\,h_1(u)\,du-\int\{1-F_0(v)\}\frac{\f_{\hat F_n}(v)}{1-\hat F_n(v)}\,h_2(v)\,dv\\
&\qquad\qquad\qquad\qquad\qquad\qquad+\int\bigl\{F_0(v)-F_0(u)\bigr\}\frac{\f_{\hat F_n}(v)-\f_{\hat F_n}(u)}{\hat F_n(v)-\hat F_n(u)}\,dH(u,v)\\
&=\int \{F_0(u)-\hat F_n(u)\}\frac{\f_{\hat F_n}(u)}{\hat F_n(u)}\,h_1(u)\,du
+\int\{F_0(v)-\hat F_n(v)\}\frac{\f_{\hat F_n}(v)}{1-\hat F_n(v)}\,h_2(v)\,dv\\
&\qquad+\int\bigl\{F_0(v)-F_0(u)-\hat F_n(v)+\hat F_n(u)\bigr\}\frac{\f_{\hat F_n}(v)-\f_{\hat F_n}(u)}{\hat F_n(v)-\hat F_n(u)}\,h(u,v)\,du\,dv\\
&=\int\left[\int_{x\le u}\frac{\f_{\hat F_n}(u)}{\hat F_n(u)}\,h_1(u)\,du+\int_{x\le v}\frac{\f_{\hat F_n}(v)}{1-\hat F_n(v)}\,h_2(v)\,dv\right.\\
&\qquad\qquad\qquad\qquad\qquad\qquad\left.+\int_{u<x\le v}\frac{\f_{\hat F_n}(v)-\f_{\hat F_n}(u)}{\hat F_n(v)-\hat F_n(u)}\,h(u,v)\,du\,dv\right]\,d\bigl(F_0-\hat F_n\bigr)(x)
\end{align*}

So this time we want to solve (approximately) the equation
\begin{align*}
&\int_{x\le u}\left\{\frac{\f_{\hat F_n}(u)}{\hat F_n(u)}\,h_1(u)+\frac{\f_{\hat F_n}(u)}{1-\hat F_n(u)}\,h_2(u)\right\}\,du+\int_{u<x\le v}\frac{\f_{\hat F_n}(v)-\f_{\hat F_n}(u)}{\hat F_n(v)-\hat F_n(u)}\,h(u,v)\,du\,dv\\
&=\k_{F_0}(x).
\end{align*}
Differentiating w.r.t.\ $x$, we get the equation
\begin{align*}
&-\frac{\f_{\hat F_n}(x)}{\hat F_n(x)}h_1(x)-\frac{\f_{\hat F_n}(x)}{1-\hat F_n(x)}h_2(x)
-\int\frac{\f_{\hat F_n}(x)-\f_{\hat F_n}(u)}{\hat F_n(x)-\hat F_n(u)}\,h(u,x)\,du\\
&\qquad\qquad\qquad\qquad\qquad\qquad\qquad\qquad+\int\frac{\f_{\hat F_n}(v)-\f_{\hat F_n}(x)}{\hat F_n(v)-\hat F_n(x)}\,h(x,v)\,dv\\
&=\k'_{F_0}(x).
\end{align*}

As before, we can switch again to the representation $\f(x)=\int_{y\in(x,M]}a(y)\,d\hat F_n(y)$ of integrated score functions $a$ and use $\int_{y\in(x,M]}a(y)\,d\hat F_n(y)=-\int_{y\in[0,x]}a(y)\,d\hat F_n(y)$. This yields the preceding equation with minus and plus signs reversed on the left side.
Now we get instead of the equations (\ref{matrix_eq}) the equations:
\begin{align}
\label{matrix_eq2}
&y_i\left\{\frac{\dd_i(h_1)}{\hat F_n(t_i)}+\frac{\dd_i(h_2)}{(1-\hat F_n(t_i))}+\sum_{j<i}\frac{\dd_{ji}(h)}{\hat F_(t_i)-\hat F_n(t_j)}+\sum_{j>i}\frac{\dd_{ij}(h)}{\hat F_n(t_j)-\hat F_n(t_i)}\right\}\nonumber\\
&=\k_{F_0}(t_{i+1})-\k_{F_0}(t_i)+\sum_{j<i}\frac{\dd_{ji}(h)\,y_j}{\hat F_(t_i)-\hat F_n(t_j)}+\frac{\dd_{ij}(h)\,y_j}{\hat F_n(t_j)-\hat F_n(t_i)},\qquad i=1,\dots,m-1,
\end{align}
As noted in \cite{GeGr:99}, p.\ 669, this is an equation of type $A{\bmy}={\bm b}$, where $A$ is a symmetrix, strictly diagonally dominant $M$-matrix (also called a Stieltjes matrix), which means that the equation has a unique solution, see \cite{Berman_Plemmons:79}. This implies that the right-continuous piecewise constant function $\f_{\hat F_n}$, constant on the same intervals as $\hat F_n$, is uniquely determined. Of course the same considerations hold for the equations (\ref{matrix_eq}). 

In contrast with the situation for the LS estimator, the function $\phi_{\hat F_n}$ now has an interpretation as an integrated score function, but we do not have an explicit expression for it (see below). For the example, discussed above for the LS estimator, we computed $\f_{\hat F_n}$ for a sample of size $n=1000$, using the matrix equation (\ref{matrix_eq2}), where this time $\hat F_n$ is the nonparametric MLE.

\begin{figure}[!ht]
\centering
\includegraphics[width=0.5\textwidth]{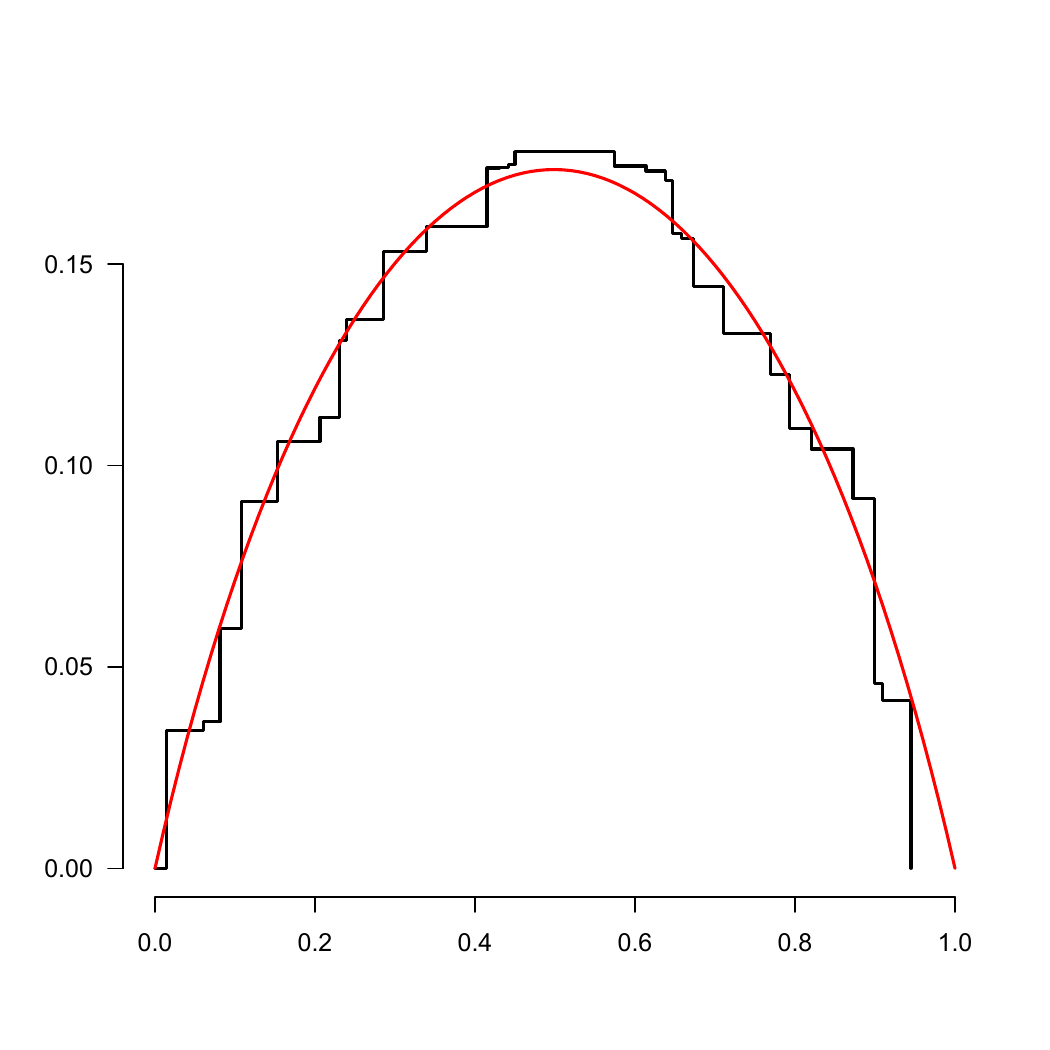}
\caption{The function $\f_{\hat F_n}$, solving equation (\ref{matrix_eq2}) for the nonparametric MLE, for a sample of size $n=1000$, where $F_0$ is the uniform distribution function on $[0,1]$ and $(U,V)$ is uniformly distributed on the upper triangle of the unit square. The red curve gives the solution of the corresponding integral equation for the underlying model.}
\label{figure:phi}
\end{figure}

We simulated the mean for the nonparametric MLE and the least squares estimators, minimizing (\ref{LS_criterion_IC}) and (\ref{LS_criterion_IC2a}), respectively, using the representation
\begin{align*}
m(F)=\int_0^1 (1-F(x))\,dx=\sum_{i=0}^m (1-F(t_i))\bigl(t_{i+1}-t_i\bigr),
\end{align*}
where $t_1<\dots< t_m$ are the points of jump of $F$, and $t_0=0$ and $t_{m+1}=1$.
For this situation the information lower bound of the asymptotic variance was numerically computed to be approximately $0.1198987$, see Section 4.2 of \cite{GeGr:99}. There is no explicit expression for it since it depends on the solution of an integral equation, for which there is no explicit form either. Based on $10,000$ simulations for sample size $1000$, we got as estimates of the asymptotic values of $n\text{var}(m(\hat F_n))$ the value $0.1237086$ for the nonparametric MLE, and the values $0.1240233$ and $0.1271471$ for the least squares estimators, minimizing (\ref{LS_criterion_IC}) and (\ref{LS_criterion_IC2a}), respectively.

All estimates give values reasonably close to the information lower bound and the least squares estimates also give $\sqrt{n}$-consistent estimates of the smooth functionals. The boxplots for the three estimates are shown in Figure \ref{figure:boxplot}, based on $10,000$ samples of size $n=1000$.

\begin{figure}[!ht]
\begin{subfigure}[b]{0.32\textwidth}
\includegraphics[width=\textwidth]{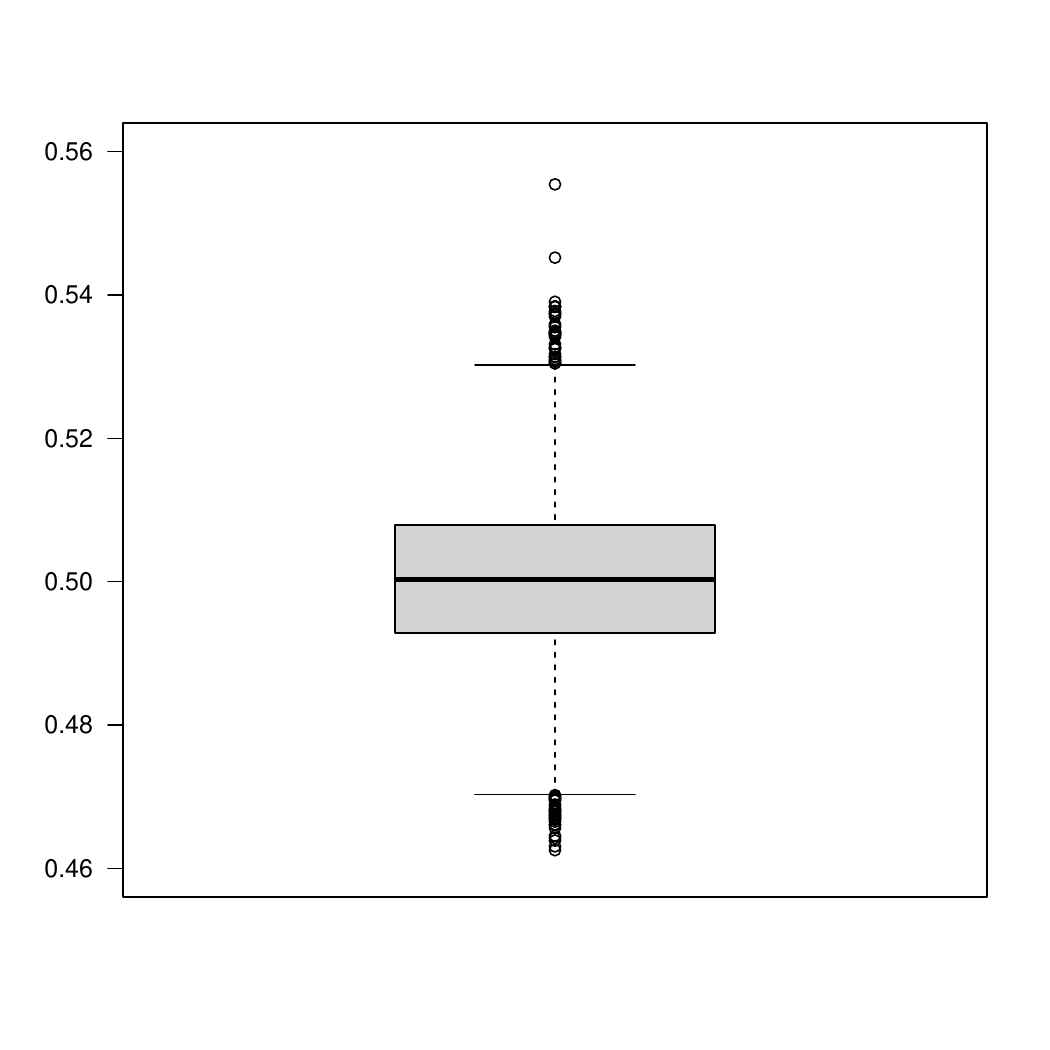}
\caption{}
\label{fig:oxPlot_mean_MLE}
\end{subfigure}
\begin{subfigure}[b]{0.32\textwidth}
\includegraphics[width=\textwidth]{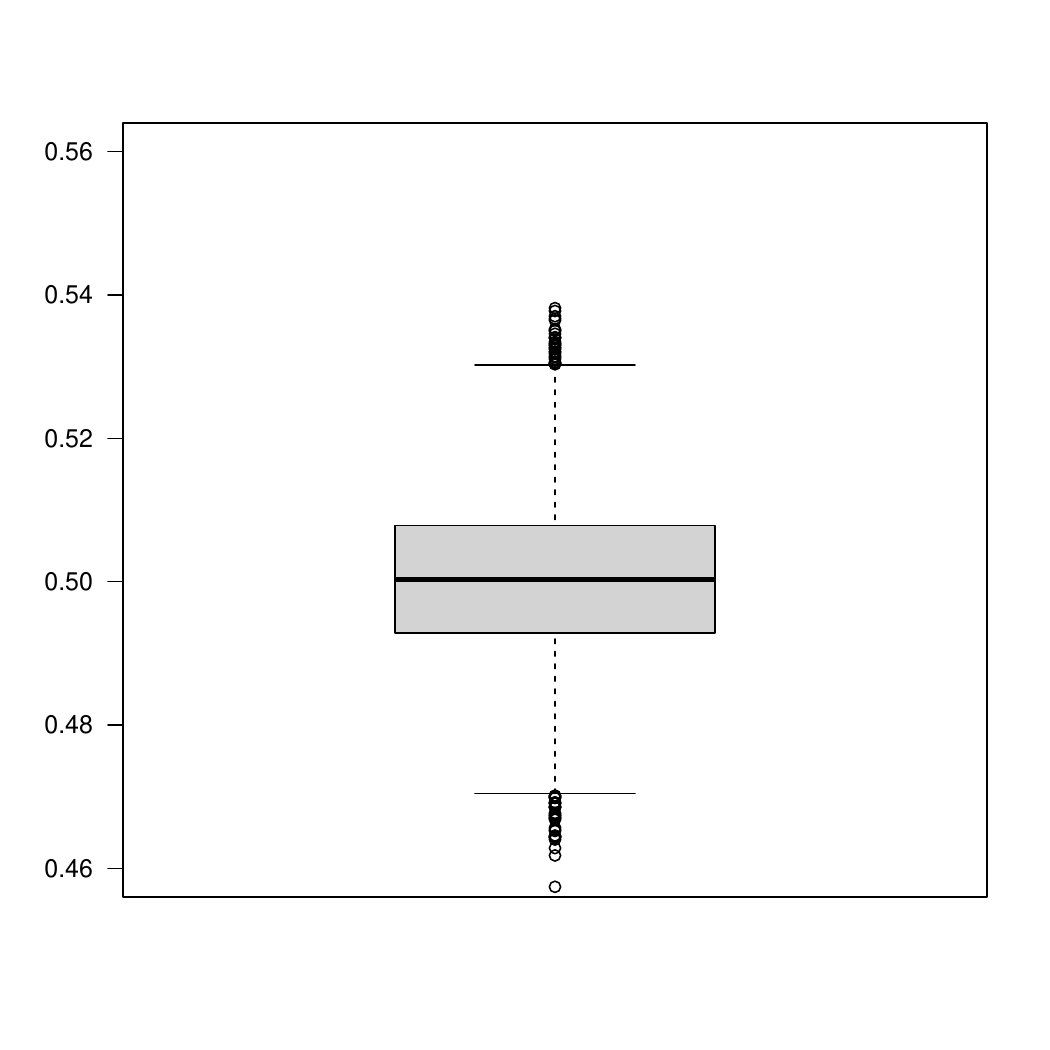}
\caption{}
\label{fig:BoxPlot_mean_LS}
\end{subfigure}
\begin{subfigure}[b]{0.32\textwidth}
\includegraphics[width=\textwidth]{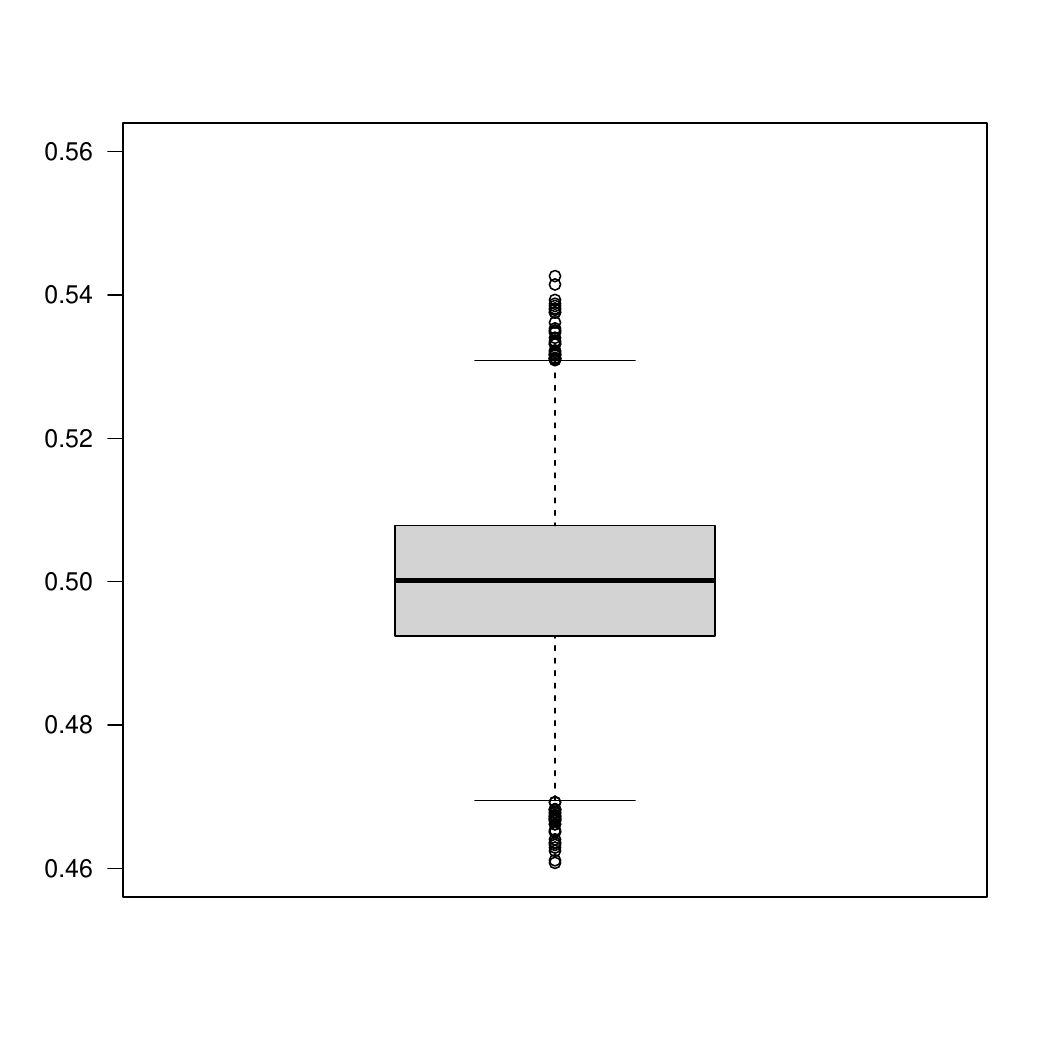}
\caption{}
\label{fig:BoxPlot_mean_simple_LS}
\end{subfigure}
\caption{(a) Mean estimates of MLE (b) Mean estimates of LS estimates, based on (\ref{LS_criterion_IC}) (c) Mean estimates of LS estimates, based on (\ref{LS_criterion_IC2a}). All boxplots are based on $10,000$ samples of size $n=1000$.}
\label{figure:boxplot}
\end{figure}
 It is clear that the differences of the three estimators are rather small in estimating the first moment. It should be noted, though, that the finer distinctions in the pointwise behavior of the estimators are ``washed away'' in the estimation of the smooth functionals.

\section{Conclusion}
\label{section:conclusion}
We stated a limit result for the least squares estimator minimizing (\ref{LS_criterion_IC}) for interval censoring, case 2, in the ``non-separated case'' (``observation intervals can be arbitrariy small''), in which situation the limit behavior of the nonparametric MLE is still unknown. This is Theorem  \ref{th:limit_LS} in section \ref{sec:LS}. For the separated case the limit distribution of the nonparametric MLE was derived in \cite{piet:96}. One could say (as has been done) that the separated case is the more important case (``there will always be a positive time interval between examination times of the doctor''), but the situation where one does not have the limit for the non-separated case is rather unsatisfactory. After all, allowing arbitrarily small intervals is mostly a matter of type of scaling one allows.

There is a conjecture that the nonparametric MLE has a faster rate of convergence in section 5.2 of part 2 of \cite{GrWe:92} and there also has been the  construction of a histogram-type estimator with a faster local rate in \cite{lucien:99}, but this conjecture does not show up for the sample sizes $n=1000$ and $n=10,000$ in our simulations. See, in particular, Figure \ref{figure:MLE+LS_variances10,000} for sample size $n=10,000$ in the present paper. But it is conceivable that the $(n\log n)^{1/3}$-rate starts showing up for astronomical sample sizes. It is true that for sample size $n=10,000$ its variance has become smaller than that of the LS estimate, reversing the relation between these variances for $n=1000$.

For the least squares estimator minimizing (\ref{LS_criterion_IC}) one can extend the result to cover the whole range of situations from separated to non-separated. The least squares estimator minimizing (\ref{LS_criterion_IC}) seems generally to have a smaller (pointwise) variance than the nonparametric MLE for moderate sample sizes (like $n=1000$).

We also discussed a simpler least squares estimatior, minimizing (\ref{LS_criterion_IC2a}) which does not have to be calculated by an iterative procedure and also does not need Lagrange multipliers to keep the values of the estimator inside the interval $[0,1]$. But this estimator is clearly inferior to the least squares estimator minimizing (\ref{LS_criterion_IC}) in the pointwise estimation problem. In estimating smooth functionals, like the first moment, the estimators seem to have a rather similar behavior, as shown in Section \ref{section:smooth_functionals}

Algorithms for computing the least squares estimator were discussed and in Section \ref{section:smooth_functionals} we provide smooth functional theory needed to show that the so-called ``off-diagonal terms'' in the derivation of the asymptotic distribution are negligible. 
For the well-known case of current status data (the simplest case of interval censoring)  the least squares approach coincides with the aproach  based on analyzing the likelihood, since in that case the least square estimate and the nonparametric MLE are the same.

{\tt C++} routines with makefiles to check the algorithms and simulations are given in \cite{piet_intcens:25}. They use the iterative convex minorant algorithm.

\section{Appendix}
\label{section:appendix}
\subsection{Variance and drift of the relevant limit process for Theorem \ref{th:limit_LS}}
\label{Appendix1}
Let, for a distribution function $F$, the functions $\psi_{1,F}$ and  $\psi_{2,F}$ defined by
\begin{align}
\label{def_psi1}
\psi_{1,F}(u,v,\d_0,\d_1)=\d_0-F(u)-\{\d_1-F(v)+F(u)\}
\end{align}
and
\begin{align}
\label{def_psi2}
\psi_{2,F}(u,v,\d_0,\d_1)=\d_1-\{F(v)-F(u)\}+\{\d_0+\d_1-F(v)\}.
\end{align}

Then we have, for fixed $t_0\in(0,M)$, and $W_{n,\hat F_n}$ is defined by (\ref{def_W_F}) 
\begin{align*}
&W_{n,\hat F_n}(t_0+n^{-1/3}t)-W_{n,\hat F_n}(t_0)=\int_{s\in[t_0,t_0+n^{-1/3}t]}\,dW_{n,\hat F_n}(s)\\
&=\int_{u\in[t_0,t_0+n^{-1/3}t]}\psi_{1,\hat F_n}(u,v,\d_0,\d_1)\,d\Q_n+\int_{v\in[t_0,t_0+n^{-1/3}t]}\psi_{2,\hat F_n}(u,v,\d_0,\d_1)\,d\Q_n.
\end{align*}
Furthermore,
\begin{align}
\label{expansion_W_F}
&\int_{u\in[t_0,t_0+n^{-1/3}t]}\psi_{1,\hat F_n}(u,v,\d_0,\d_1)\,d\Q_n
+\int_{v\in[t_0,t_0+n^{-1/3}t]}\psi_{2,\hat F_n}(u,v,\d_0,\d_1)\,d\Q_n\nonumber\\
&=\int_{u\in[t_0,t_0+n^{-1/3}t]}\left\{\psi_{1,\hat F_n}(u,v,\d_0,\d_1)-\f_{1,F_0}(u,v,\d_0,\d_1)\right\}\,d\Q_n\nonumber\\
&\qquad+\int_{u\in[t_0,t_0+n^{-1/3}t]}\left\{\psi_{2,\hat F_n}(u,v,\d_0,\d_1)-\psi_{2,F_0}(u,v,\d_0,\d_1)\right\}\,d\Q_n\nonumber\\
&\qquad +X_n(t),
\end{align}
where
\begin{align}
\label{def_X_n}
&X_n(t)\nonumber\\
 &=\int_{u\in[t_0,t_0+n^{-1/3}t]}\psi_{1,F_0}(u,v,\d_0,\d_1)\,d\Q_n
 +\int_{v\in[t_0,t_0+n^{-1/3}t]}\psi_{2,F_0}(u,v,\d_0,\d_1)\,d\Q_n\nonumber\\
 &=\int_{u\in[t_0,t_0+n^{-1/3}t]}\psi_{1,F_0}(u,v,\d_0,\d_1)\,d\bigl(\Q_n-Q_0\bigr)\nonumber\\
 &\qquad\qquad\qquad\qquad + \int_{v\in[t_0,t_0+n^{-1/3}t]}\psi_{2,F_0}(u,v,\d_0,\d_1)\,d\bigl(\Q_n-Q_0\bigr),
\end{align}
where $Q_0$ is the underlying probability measure for the $(U_i,V_i,\dd_{i0},\dd_{i1})$.
We also define $X_n(t)$ for $t<0$ by changing the interval $[t_0,t_0+n^{-1/3}t]$ to $[t_0+n^{-1/3}t,t_0]$.
For this process, we have the following lemma.

\begin{lemma}
\label{limit_X_n}
Let the conditions of Theorem \ref{th:limit_LS} be satisfied. Then $n^{2/3}X_n$ converges in distribution, in the topology of uniform convergence on compacta, to the process
\begin{align*}
t\mapsto a_{t_0}W(t),\qquad t\in\R,
\end{align*}
where $W$ is standard two-sided Brownian motion and $a_{t_0}$ is defined by (\ref{scale_IC}).
\end{lemma}

\begin{remark}
{\rm Note the similarity with Lemma 9.4 in \cite{piet:24}, but note that we now take $F_0$ instead of $\hat F_n$ in the definition of $X_n$.
}
\end{remark}

\begin{proof}
Note that we can write (\ref{def_X_n}) in the form:
\begin{align}
\label{X}
X_n(t)&=n^{-1}\sum_{j:t_0\leq U_j\le t_0+n^{-1/3}t}\left\{\dd_{j0}-F_0(U_j)\right\}\nonumber\\
&\qquad-n^{-1}\sum_{j:t_0\leq U_j\le t_0+n^{-1/3}t}\left\{\dd_{j1}-\{F_0(V_j)-F_0(U_j))\}\right\}\nonumber\\
&\qquad+n^{-1}\sum_{j:t_0\leq V_j\le t_0+n^{-1/3}t}\left\{\dd_{j1}-\{F_0(V_j)-F_0(U_j))\}\right\}\nonumber\\
&\qquad+n^{-1}\sum_{j:t_0\leq V_j\le t_0+n^{-1/3}t}\left\{\dd_{j0}+\dd_{j1}-F_0(V_j)\}\right\}.
\end{align}
We have:
\begin{align}
\label{variance_expansion}
&n^{4/3}\text{var}\bigl(X_n(t)\bigr)\nonumber\\
&\sim tF_0(t_0)\{1-F_0(t_0)\}\bigl\{h_1(t_0)+h_2(t_0)\bigr\}\nonumber\\
&\qquad+t\int_{v=t_0}^M\{F_0(v)-F_0(t_0)\}\left[1-\{F_0(v)-F_0(t_0)\}\right]\,h(t_0,v)\,dv\nonumber\\
&\qquad+t\int_{u=0}^{t_0}\{F_0(t_0)-F_0(u)\}\left[1-\{F_0(t_0)-F_0(u)\}\right]\,h(u,t_0)\,du\nonumber\\
&\qquad+2tF_0(t_0)\int_{v=t_0}^M\{F_0(v)-F_0(t_0)\}\,h(t_0,v)\,dv\nonumber\\
&\qquad+2t\{1-F_0(t_0)\}\int_{u=0}^{t_0}\{F_0(t_0)-F_0(u)\}\,h(u,t_0)\,du,
\end{align}
where $h_1$ and $h_2$ are the marginal densities of $h$.

This is seen as follows. Conditionally on $(U_j,V_j)$, the triple $(\dd_{j0},\dd_{j1},\dd_{j2})$ has a trinomial distribution, with parameters
\begin{align*}
p=F_0(U_j),\qquad q=F_0(V_j)-F_0(U_j),\qquad 1-p-q=1-F_0(V_j).
\end{align*}
This yields the first three terms on the right of (\ref{variance_expansion}), since the conditional variances are of the form $p(1-p)$, $q(1-q)$ and $(p+q)(1-p-q)$. The last two terms on the right of (\ref{variance_expansion}) are due to the covariance of $\dd_{j0}$ and $-\dd_{j1}$ in the first two terms and $\dd_{j1}$ and $-\dd_{j2}$ in the last two term on the right-hand side of (\ref{X}), respectively. Other covariances are of lower order, since these involve values of $U_j$ and $V_j$ in the shrinking interval $[t_0,t_0+n^{-1/3}t]$.
So we get
\begin{align*}
n^{2/3}X_n\stackrel{{\cal D}}\longrightarrow a_{t_0}W,\qquad n\to\infty.
\end{align*}
from tightness and the central limit theorem.
\end{proof}

\subsection{Proof of Theorem \ref{th:limit_LS}}
\label{Appendix1a}
\begin{proof}
We now generalize the approach in the proof of Theorem 3.7 in \cite{piet_geurt:14}. Let $\G_n$ be defined by
\begin{align*}
\G_n(t)=2\int_{u\le t}\,d\H_n(u,v)+2\int_{v\le t}\,d\H_n(u,v),
\end{align*}
and let $V_n$ be defined by
\begin{align*}
V_n(t)=\int_{w\le t}\hat F_n(w)\,d\G_n(w)+W_{n,\hat F_n}(t),\qquad t\ge0,
\end{align*}
where $W_{n,\hat F_n}$ is defined by (\ref{def_W_F}). The values of the processes are zero for $t\le0$.

 Then $\hat F_n$ is the left-continuous slope of the greatest convex minorant of cumulative sum (``cusum'') diagram
\begin{align}
\label{cusum_diagram}
(\G_n(t),V_n(t)),\qquad t\ge0,
\end{align}
where (\ref{cusum_diagram}) runs through the points $(\G_n(U_i),V_n(U_i))$ and $\G_n(V_i),V_n(V_i))$, $i=1,\dots,n$, and the slopes are put equal to zero if they are less than zero and equal to $1$ if they are larger than $1$.

We next define, for $a\in(0,1)$,
\begin{align*}
U_n(a)=\text{argmin}\left\{t\in\R:V_n(t)-a\G_n(t)\right\},
\end{align*}
where we take the supremum in case of multiple argmins and use, for $a_0=F_0(t_0)$, the ``switch relation''
\begin{align*}
P\left\{n^{1/3}\bigl\{\hat F_n(t_0)-F_0(t_0)\bigr\}\ge x\right\}=P\left\{U_n(a_0+n^{-1/3}x)\le t_0\right\},
\end{align*}
see Figure 3.7 on p.\ 69 of \cite{piet_geurt:14}. We have:
\begin{align*}
&U_n(a_0+n^{-1/3}x)-t_0=\text{argmin}\left\{t\in\R:V_n(t)-(a_0+n^{-1/3}x)\G_n(t)\right\}-t_0\\
&=\text{argmin}\left\{t_0+n^{-1/3}t\in\R:V_n(t_0+n^{-1/3}t)-(a_0+n^{-1/3}x)\G_n(t_0+n^{-1/3}t)\right\}-t_0\\
&=n^{-1/3}\text{argmin}\left\{t\in\R:V_n(t_0+n^{-1/3}t)-(a_0+n^{-1/3}x)\G_n(t_0+n^{-1/3}t)\right\}.
\end{align*}
So we can conclude
\begin{align*}
&n^{1/3}\bigl\{U_n(a_0+n^{-1/3}x)-t_0\bigr\}\\
&=\text{argmin}\left\{t\in\R:V_n(t_0+n^{-1/3}t)-(a_0+n^{-1/3}x)\G_n(t_0+n^{-1/3}t)\right\},
\end{align*}
and we have to determine
\begin{align*}
P\left\{n^{1/3}\bigl\{U_n(a_0+n^{-1/3}x)-t_0\bigr\}\le0\right\}.
\end{align*}
Since the argmin function does not change by adding a constant to the argument or multiplying the argument by a constant, we can replace $n^{1/3}\left\{U_n(a_0+n^{-1/3}x)-t_0\right\}$ by the argmin (in $t$) of:
\begin{align*}
n^{2/3}\bigl\{V_n(t_0+n^{-1/3}t)-V_n(t_0)\bigr\}-n^{2/3}(a_0+n^{-1/3}x)\bigl\{\G_n(t_0+n^{-1/3}t)-\G_n(t_0)\bigr\}.
\end{align*}

We get, if $t\ge0$,
\begin{align*}
&V_n(t_0+n^{-1/3}t)-V_n(t_0)-a_0\left\{\G_n(t_0+n^{-1/3}t)-\G_n(t_0)\right\}\\
&=\int_{w\in(t_0,t_0+n^{-1/3}t]}\left\{\hat F_n(w)-F_0(t_0)\right\}\,d\G_n(w)+W_{n,\hat F_n}(t_0+n^{-1/3}t)-W_{n,\hat F_n}(t_0)\\
&=\int_{w\in(t_0,t_0+n^{-1/3}t]}\left\{F_0(w)-F_0(t_0)\right\}\,d\G_n(w)+X_n(t)\\
&\qquad-\int_{u\in[t_0,t_0+n^{-1/3}t]}\left\{F_0(v)-\hat F_n(v)\right\}\,d\H_n(u,v)\nonumber\\
&\qquad-\int_{v\in[t_0,t_0+n^{-1/3}t]}\left\{F_0(u)-\hat F_n(u)\right\}\,d\H_n(u,v),
\end{align*}
where $X_n$ is defined by (\ref{X}).
The last two terms are smooth functionals of the model and can be shown to be of order $O_p(n^{-5/6})$. We discuss this matter  in section \ref{appendix2} and Section \ref{section:smooth_functionals}.

So we get:
\begin{align}
\label{prelimit}
&n^{2/3}\left\{V_n(t_0+n^{-1/3}t)-V_n(t_0)\right\}-n^{2/3}(a_0+n^{-1/3}x)\left\{\G_n(t_0+n^{-1/3}t)-\G_n(t_0)\right\}\nonumber\\
&=n^{2/3}\int_{w\in(t_0,t_0+n^{-1/3}t]}\left\{F_0(w)-F_0(t_0)\right\}\,d\G_n(w)+n^{2/3}X_n(t)\nonumber\\
&\qquad\qquad\qquad\qquad-n^{1/3}x\left\{\G_n(t_0+n^{-1/3}t)-\G_n(t_0)\right\}
+O_p\left(n^{-1/6}\right).
\end{align}
Since
\begin{align*}
&n^{2/3}\int_{u\in[t_0,t_0+n^{-1/3}t]}\{F_0(w)-F_0(t_0)\}\,d\G_n(w)
\sim f_0(t_0)\{h_1(t_0)+h_2(t_0)\}t^2,
\end{align*}
the right-hand side of (\ref{prelimit}) converges in distribution, in the topology of uniform convergence on compacta, to the process
\begin{align*}
t\mapsto a_{t_0}W(t)+f_0(t_0)\{h_1(t_0)+h_2(t_0)\}t^2-2\{h_1(t_0)+h_2(t_0)\}x,\qquad t\ge0.
\end{align*}
We get a similar result for $t\le0$, we and get therefore:
\begin{align*}
&P\left[n^{1/3}\left\{U_n(a_0+n^{-1/3}x)-t_0\right\}\le0\right]\\
&\longrightarrow P\left[\text{argmin}_t\left\{a_{t_0}W(t)+f_0(t_0)\{h_1(t_0)+h_2(t_0)\}t^2-2\{h_1(t_0)+h_2(t_0)\}x\right\}\le0\right].
\end{align*}
Next note that
 \begin{align*}
& P\left[\text{argmin}_t\left\{a_{t_0}W(t)+f_0(t_0)\{h_1(t_0)+h_2(t_0)\}t^2-2\{h_1(t_0)+h_2(t_0)\}x\right\}\le0\right]\\
&= P\left[\text{argmin}_t\left\{a_{t_0}W(t)+f_0(t_0)\{h_1(t_0)+h_2(t_0)\}\left(t-\frac{x}{f_0(t_0)}\right)^2\right\}\le0\right]\\
\end{align*}
Defining
\begin{align*}
T(x)=\text{argmin}_t\left\{a_{t_0}W(t)+f_0(t_0)\{h_1(t_0)+h_2(t_0)\}\left(t-\frac{x}{f_0(t_0)}\right)^2\right\}-\frac{x}{f_0(t_0)}\,,
\end{align*}
we have that the process
\begin{align*}
x\mapsto T(x),\qquad x\in\R,
\end{align*}
is stationary (this is one of the key results of \cite{gro:89}), so we have:
\begin{align*}
&P\left[\text{argmin}_t\left\{a_{t_0}W(t)+f_0(t_0)\{h_1(t_0)+h_2(t_0)\}\left(t-\frac{x}{f_0(t_0)}\right)^2\right\}\le0\right]\\
&=P\left\{T(x)\le0\right\}=P\left\{T(0)\le-\frac{x}{f_0(t_0)}\right\}=P\left\{T(0)\ge\frac{x}{f_0(t_0)}\right\},
\end{align*}
where the last equality follows from. the symmetry of the distribution of $T(0)$ around zero (compare with Theorem 5.2, p.
  95 in part II of \cite{GrWe:92}).
  
 The result of Theorem \ref{th:limit_LS} now follows from Brownian scaling,
 \end{proof}

\subsection{Negligibility of the ``off-diagonal terms''}
\label{appendix2}
In subsection \ref{Appendix1a} we met two terms which we call ``off-diagonal terms'', because, for example in the interior point method for solving the minimization problem, they correspond to non-zero off-diagonal terms of the Hessian. This is a difficulty we do not have to deal with in the current status model.  They are the terms
\begin{align}
\label{off-diag1}
\int_{u\in[t_0,t_0+n^{-1/3}t]}\bigl\{F_0(v)-\hat F_n(v)\bigr\}\,d\H_n(u,v)
\end{align}
and
\begin{align}
\label{off-diag2}
\int_{v\in[t_0,t_0+n^{-1/3}t]}\bigl\{F_0(u)-\hat F_n(u)\bigr\}\,d\H_n(u,v).
\end{align}
In (\ref{off-diag1}) the variable $v$ roughly varies between $t_0$ and $M$, since we have $v\ge u$. Likewise, in in (\ref{off-diag2}) the variable $u$ roughly varies between $0$ and $t_0$.

We now take a closer look at (\ref{off-diag1}). First of all, we can write
\begin{align*}
&\int_{u\in[t_0,t_0+n^{-1/3}t]}\bigl\{F_0(v)-\hat F_n(v)\bigr\}\,d\H_n(u,v)\\
&=\int_{u\in[t_0,t_0+n^{-1/3}t]}\bigl\{F_0(v)-\hat F_n(v)\bigr\}\,d\bigl(\H_n-H\bigr)(u,v)\\
&\qquad+\int_{u\in[t_0,t_0+n^{-1/3}t]}\bigl\{F_0(v)-\hat F_n(v)\bigr\}\,dH(u,v).
\end{align*}
For the first term on the right we have:
\begin{align*}
&\int_{u\in[t_0,t_0+n^{-1/3}t]}\bigl\{F_0(v)-\hat F_n(v)\bigr\}\,d\bigl(\H_n-H\bigr)(u,v)\\
&=\int_{u\in[t_0,t_0+n^{-1/3}t]}F_0(v)\,d\bigl(\H_n-H\bigr)(u,v)
-\int_{u\in[t_0,t_0+n^{-1/3}t]}\hat F_n(v)\,d\bigl(\H_n-H\bigr)(u,v),
\end{align*}
where both terms are of order $O_p(n^{-5/6})$. This is obvious for the first term, and for the second term if follows from the entropy with bracketing for bounded monotone functions.

So we turn to the term
\begin{align*}
\int_{u\in[t_0,t_0+n^{-1/3}t]}\bigl\{F_0(v)-\hat F_n(v)\bigr\}\,dH(u,v).
\end{align*}
This term is also expected to be of order $O_p(n^{-5/6})$.
The key to this is the fact that integrals of the form
\begin{align*}
\int_a^b\bigl\{F_0(x)-\hat F_n(x)\bigr\}\,dx
\end{align*}
for $b>a$ are of order $O_p(n^{-1/2})$. This follows from smooth functional theory, discussed in Section \ref{section:smooth_functionals}. We cannot work with simple upper bounds for the distance between $\hat F_n$ and $F_0$, since these will only give us order $O_p(n^{-2/3})$ instead of order $O_p(n^{-5/6})$.

 \subsection{Theorem \ref{th:limit_LS2}}
\label{appendix_teorem2}
The proof proceeds along the lines of the proof of Theorem \ref{th:limit_LS}, but is much simpler, because we do not have to deal with ``off-diagonal terms'' and can compute the estimator in one step. Iterations, as in the situation of Theorem \ref{th:limit_LS}, are not needed for the computation and the cusum diagram is not ``self-induced''. As noted above, the price we have to pay for this simplicity, though, is that the estimator is less efficient.

We start by analyzing the covariance structure of the process $W_{n,F_0}^{(2)}$, defined by (\ref{def_W_F2}). Analogously to the proof of Theorem \ref{th:limit_LS}, we define the process $X_n^{(2)}$ by
\begin{align*}
X_n^{(2)}(t)=W_{n,F_0}^{(2)}(t_0+n^{-1/3}t)-W_{n,F_0}^{(2)}(t_0).
\end{align*}
For the process $X_n^{(2)}$ we have the following lemma.

\begin{lemma}
\label{limit_X_n2}
Let the conditions of Theorem \ref{th:limit_LS} be satisfied. Then $n^{2/3}X_n^{(2)}$ converges in distribution, in the topology of uniform convergence on compacta, to the process
\begin{align*}
t\mapsto a_{t_0}'W(t),\qquad t\in\R,
\end{align*}
where $W$ is standard two-sided Brownian motion and $a_{t_0}'$ is defined by (\ref{scale_IC2}).
\end{lemma}

\begin{proof}
Let $t\ge0$. Then:
\begin{align}
\label{variance_X_2}
X_n^{(2)}(t)
&=n^{-1}\sum_{j:t_0\leq U_j\le t_0+n^{-1/3}t}\left\{\dd_{j0}-F_0(U_j)\right\}\nonumber\\
&\qquad\qquad\qquad\qquad+n^{-1}\sum_{j:t_0\leq V_j\le t_0+n^{-1/3}t}\left\{\dd_{j0}+\dd_{j1}-\{F_0(V_j)\}\right\}.
\end{align}
This means:
\begin{align*}
&n^{4/3}\text{var}\left(X_n^{(2)}(t)\right)\sim F_0(t_0)\{1-F_0(t_0)\}\{h_1(t_0)+h_2(t_0)\}.
\end{align*}
The covariance between $\dd_{j0}$ and $\dd_{j0}+\dd_{j1}$ in the first and second line of (\ref{variance_X_2}) gives a contribution of lower order, since this involves a $U_j$ and $V_j$ both lying in the shrinking interval $[t_0,t_0+n^{-1/3}t]$, which has a probability tending to zero, as $n\to\infty$.

So the result follows again from tightness and the central limit theorem.
\end{proof}

\begin{proof}[Proof of Theorem \ref{th:limit_LS2}]
This time we define
\begin{align*}
\G_n(t)=\int_{u\le t}\,d\H_n(u,v)+\int_{v\le t}\,d\H_n(u,v),
\end{align*}
and, as in the proof of Theorem \ref{th:limit_LS}, $V_n$ by
\begin{align*}
V_n(t)=\int_{w\le t}\hat F_n(w)\,d\G_n(w)+W_{n,\hat F_n}^{(2)}(t),\qquad t\ge0.
\end{align*}
The values of the processes are zero for $t\le0$. Note that $V_n$ can this time be written in the form
\begin{align*}
V_n(t)=\int_{w\le t}F_0(w)\,d\G_n(w)+W_{n,F_0}^{(2)}(t),\qquad t\ge0,
\end{align*}
so $\hat F_n$ is not part of this representation of $V_n$.

We next define, for $a\in(0,1)$,
\begin{align*}
U_n(a)=\text{argmin}\left\{t\in\R:V_n(t)-a\G_n(t)\right\},
\end{align*}
taking the supremum in case of multiple argmins and use, for $a_0=F_0(t_0)$, the switch relation
\begin{align*}
P\left\{n^{1/3}\bigl\{\hat F_n(t_0)-F_0(t_0)\bigr\}\ge x\right\}=P\left\{U_n(a_0+n^{-1/3}x)\le t_0\right\}.
\end{align*}
We have again
\begin{align*}
&n^{1/3}\left\{U_n(a_0+n^{-1/3}x)-t_0\right\}\\
&=\text{argmin}\left\{t\in\R:V_n(t_0+n^{-1/3}t)-(a_0+n^{-1/3}x)\G_n(t_0+n^{-1/3}t)\right\},
\end{align*}
and we have to determine
\begin{align*}
P\left\{n^{1/3}\left\{U_n(a_0+n^{-1/3}x)-t_0\right\}\le0\right\}.
\end{align*}

We get, if $t\ge0$,
\begin{align*}
&V_n(t_0+n^{-1/3}t)-V_n(t_0)-a_0\left\{\G_n(t_0+n^{-1/3}t)-\G_n(t_0)\right\}\\
&=\int_{w\in(t_0,t_0+n^{-1/3}t]}\left\{F_0(w)-F_0(t_0)\right\}\,d\G_n(w)+X_n^{(2)}(t).
\end{align*}

Hence:
\begin{align}
\label{prelimit2}
&n^{2/3}\left\{V_n(t_0+n^{-1/3}t)-V_n(t_0)\right\}-n^{2/3}(a_0+n^{-1/3}x)\left\{\G_n(t_0+n^{-1/3}t)-\G_n(t_0)\right\}\nonumber\\
&=n^{2/3}\int_{w\in(t_0,t_0+n^{-1/3}t]}\left\{F_0(w)-F_0(t_0)\right\}\,d\G_n(w)+n^{2/3}X_n^{(2)}(t)\nonumber\\
&\qquad\qquad\qquad\qquad\qquad\qquad\qquad\qquad-n^{1/3}x\left\{\G_n(t_0+n^{-1/3}t)-\G_n(t_0)\right\}.
\end{align}
Since
\begin{align*}
&n^{2/3}\int_{u\in[t_0,t_0+n^{-1/3}t]}\{F_0(w)-F_0(t_0)\}\,d\G_n(w)
\sim \tfrac12f_0(t_0)\{h_1(t_0)+h_2(t_0)\}t^2,
\end{align*}
the right-hand side of (\ref{prelimit2}) converges in distribution, in the topology of uniform convergence on compacta, to the process
\begin{align*}
t\mapsto a_{t_0}'W(t)+\tfrac12f_0(t_0)\{h_1(t_0)+h_2(t_0)\}t^2-\{h_1(t_0)+h_2(t_0)\}x,\qquad t\ge0.
\end{align*}
We get a similar result for $t\le0$, we and get therefore:
\begin{align*}
&P\left[n^{1/3}\left\{U_n(a_0+n^{-1/3}x)-t_0\right\}\le0\right]\\
&\longrightarrow P\left[\text{argmin}_t\left\{a_{t_0}'W(t)+\tfrac12f_0(t_0)\{h_1(t_0)+h_2(t_0)\}t^2-\{h_1(t_0)+h_2(t_0)\}x\right\}\le0\right].
\end{align*}
Next note that
 \begin{align*}
& P\left[\text{argmin}_t\left\{a_{t_0}'W(t)+f_0(t_0)\{h_1(t_0)+h_2(t_0)\}t^2-\{h_1(t_0)+h_2(t_0)\}x\right\}\le0\right]\\
&= P\left[\text{argmin}_t\left\{a_{t_0}'W(t)+\tfrac12f_0(t_0)\{h_1(t_0)+h_2(t_0)\}\left(t-\frac{x}{f_0(t_0)}\right)^2\right\}\le0\right]\\
\end{align*}
From this point on the proof can be completed in the same way as the proof of Theorem \ref{th:limit_LS} with as only differences the presence of the factor $\tfrac12$ in front of $f_0(t_0)$ and the scale factor $a_{t_0}'$ instead of $a_{t_0}$.
\end{proof}

\bibliographystyle{imsart-nameyear}
\bibliography{cupbook}

\end{document}